\documentclass{amsart}
\usepackage{amsmath}
\usepackage{amsfonts}
\usepackage{enumerate}
\usepackage{amssymb}
\usepackage{amscd}
\usepackage[all]{xy}
\usepackage{bm}

\DeclareMathOperator{\gr}{gr}
\DeclareMathOperator{\Der}{Der}

\DeclareMathOperator{\ad}{ad}
\DeclareMathOperator{\SL}{SL}
\DeclareMathOperator{\GL}{GL}
\DeclareMathOperator{\Cdim}{Cdim}
\DeclareMathOperator{\ann}{ann}
\DeclareMathOperator{\Ann}{Ann}

\begin{document}
\theoremstyle{plain}
\newtheorem{MainThm}{Theorem}
\renewcommand{\theMainThm}{\Alph{MainThm}}
\newtheorem{MainCor}{Corollary}
\renewcommand{\theMainCor}{\Alph{MainCor}}
\newtheorem*{trm}{Theorem}
\newtheorem*{lem}{Lemma}
\newtheorem*{prop}{Proposition}
\newtheorem*{defn}{Definition}
\newtheorem*{thm}{Theorem}
\newtheorem*{thmA}{Theorem A}
\newtheorem*{thmB}{Theorem B}
\newtheorem*{thmC}{Theorem C}
\newtheorem*{thmAA}{Theorem A'}
\newtheorem*{thmBB}{Theorem B'}
\newtheorem*{thmCC}{Theorem C'}
\newtheorem*{example}{Example}
\newtheorem*{cor}{Corollary}
\newtheorem*{conj}{Conjecture}
\newtheorem*{hyp}{Hypothesis}
\newtheorem*{thrm}{Theorem}
\newtheorem*{quest}{Question}
\theoremstyle{remark}
\newtheorem*{rem}{Remark}
\newtheorem*{rems}{Remarks}
\newtheorem*{notn}{Notation}
\newcommand{\FFp}{K}
\newcommand{\Fp}{\mathbb{F}_p}
\newcommand{\Fq}{\mathbb{F}_q}
\newcommand{\Zp}{\mathbb{Z}_p}
\newcommand{\Qp}{\mathbb{Q}_p}
\newcommand{\Kr}{\mathcal{K}}
\newcommand{\Rees}[1]{\widetilde{#1}}
\newcommand{\invlim}{\lim\limits_{\longleftarrow}}
\newcommand{\Md}[1]{\mathcal{M}(#1)}
\newcommand{\Pj}[1]{\mathcal{P}(#1)}
\newcommand{\G}[2]{\mathcal{G}_{#1}(#2)}
\newcommand{\fm}{\mathfrak m}
\newcommand{\Hom}{\operatorname{Hom}}
\newcommand{\Frac}{\operatorname{Frac}}
\newcommand{\Kdim}{\operatorname{Kdim}}
\newcommand{\GKdim}{\operatorname{GKdim}}
\newcommand{\Ext}{\operatorname{Ext}}
\newcommand{\injdim}{\operatorname{injdim}}

\title{Reflexive ideals in Iwasawa algebras}
\author{K. Ardakov, F. Wei and J. J. Zhang}

\address{(K. Ardakov) School of Mathematical Sciences, University of Nottingham, University Park,
Nottingham, NG7 2RD, United Kingdom}

\email{Konstantin.Ardakov@nottingham.ac.uk}

\address{(F. Wei) Department of Applied Mathematics,
Beijing Institute of Technology, Beijing, 100081, P. R. China}

\email{daoshuo@bit.edu.cn}

\address{(J. J. Zhang) Department of Mathematics, Box 354350,
University of Washington, Seattle, Washington 98195, USA}

\email{zhang@math.washington.edu}

\begin{abstract}
Let $G$ be a torsionfree compact $p$-adic analytic group. We give sufficient conditions on $p$ and $G$ which ensure that the Iwasawa algebra $\Omega_G$ of $G$ has no non-trivial two-sided reflexive ideals. Consequently, these conditions imply that every nonzero normal element in $\Omega_G$ is a unit. We show that these conditions hold in the case when $G$ is an open subgroup of $\SL_2(\Zp)$ and $p$ is arbitrary. Using a previous result of the first author, 
we show that there are only two prime ideals in $\Omega_G$ when $G$ is a congruence subgroup of $\SL_2(\Zp)$: the zero ideal and the unique maximal ideal. These statements partially answer some
questions asked by the first author and Brown.
\end{abstract}



\maketitle
\let\le=\leqslant  \let\leq=\leqslant
\let\ge=\geqslant  \let\geq=\geqslant

\setcounter{section}{-1}
\section{Introduction}
\label{xxsec0}
\subsection{Motivation}
\label{xxsec0.1} The Iwasawa theory for elliptic curves in
arithmetic geometry provides the main motivation for the study of Iwasawa
algebras $\Lambda_G$, for example when $G$ is a certain subgroup of the
$p$-adic analytic group $\GL_2({\mathbb Z}_p)$ \cite[Section
8]{CSS}. Homological and ring-theoretic properties of these Iwasawa
algebras are useful for understanding the structure of the
Pontryagin dual of Selmer groups \cite{OV,V3} and other modules
over the Iwasawa algebras. Several recent papers \cite{A,AB1, AB2,
V1,V2} are devoted to ring-theoretic properties of the Iwasawa
algebras. One central question in this research direction is
whether there are any non-trivial prime ideals in $\Omega_G = \Lambda_G/p\Lambda_G$, when $G$ is an open subgroup of $\SL_2({\mathbb Z}_p)$, see \cite[Question, p.197]{A}. The aim of this paper is to answer this
question and a few other related open questions.

An Iwasawa algebra over any uniform subgroup of
$\SL_2({\mathbb Z}_p)$ is local and extremely noncommutative
since the only nonzero prime ideal is the maximal ideal by one of our main results, 
Theorem C. These algebras
give rise to a class of so-called just infinite-dimensional
algebras. On the other hand, their associated graded rings are
commutative polynomial rings and hence Iwasawa
algebras share many good properties with commutative rings. This
class of algebras is very interesting from the ring-theoretic
point of view and deserves further investigation.

\subsection{Definitions}
\label{xxsec0.2}
Throughout we fix a prime integer $p$. Let $\mathbb{Z}_p$ be
the ring of $p$-adic integers and let ${\mathbb F}_p$ be the
field ${\mathbb Z}/(p)$. We refer to the book \cite{DDMS} for
the definition and  basic properties of a $p$-adic analytic
group and related material.
Let $G$ be a compact $p$-adic analytic group.
The \emph{Iwasawa algebra} of $G$ (or the \emph{completed
group algebra} of $G$ over ${\mathbb Z}_p$) is defined
to be
\[\Lambda_G : =\lim_{\longleftarrow} \mathbb{Z}_p[G/N],\]
where the inverse limit is taken over the open normal subgroups
$N$ of $G$ \cite[p.443]{La}, \cite[p.155]{DDMS}.
A closely related algebra is $\Omega_G := \Lambda_G/p\Lambda_G$, whose alternative definition is
\[\Omega_G : =\lim_{\longleftarrow} {\mathbb F}_p[G/N].\]
For simplicity, the algebra $\Omega_G$ is also called the
{\it Iwasawa algebra} of $G$ (or the {\it completed
group algebra} of $G$ over $\Fp$). We refer to \cite{AB1}
for some basic properties of $\Lambda_G$ and $\Omega_G$
and to the articles \cite{CSS,CFKSV,V1,V2} for general readings
about Iwasawa algebras and their modules.

In this paper, we deal entirely with $\Omega_G$. For a treatment of 
the implications of our results for the Iwasawa algebra $\Lambda_G$, see 
\cite{A2}.

\subsection{Reflexive ideals}
\label{IntroRefl} Let $A$ be any algebra and $M$ be a left
$A$-module. We call $M$ {\it reflexive} if the canonical map
$$M\to \Hom_{A^{\sf op}}(\Hom_{A}(M,A),A)$$
is an isomorphism. A reflexive right $A$-module is defined
similarly. We will call a two-sided ideal $I$ of $A$ \emph{reflexive}
if it is reflexive as a right and as a left $A$-module.

For the rest of the introduction we assume that $G$ is torsionfree,
in which case $\Omega_G$ is an Auslander regular domain.
Here is our first main result.

\begin{thmA}
Let $G$ be a torsionfree compact $p$-adic
analytic group whose $\Qp$-Lie algebra $\mathcal{L}(G)$ is split semisimple over
$\Qp$. Suppose that $p \geq 5$ and that $p \nmid n$ in the case when $\mathfrak{sl}_n(\Qp)$ occurs as a direct 
summand of $\mathcal{L}(G)$. Then $\Omega_G$ has no non-trivial two-sided reflexive
ideals. 
\end{thmA}

The proof of Theorem A is based on a result from \cite{AWZ}. For
a few small $p$, there are some extra difficulties to be dealt with; hence we exclude these primes from
consideration. We believe that these restrictions on $p$ are not really necessary.

\subsection{Normal elements}
\label{xxsec0.4}
Recall that an element $w$ of a ring $A$ is said to be \emph{normal} if $wA = Aw$. The first author and Brown \cite[Question K]{AB1} asked whether under hypotheses on $G$ similar to the ones in Theorem A any nonzero normal element of $\Omega_G$ must be a unit. Because every nonzero normal element $w \in \Omega_G$ gives rise to a nonzero reflexive two-sided ideal $w\Omega_G$, Theorem A implies 

\begin{thmB}
Under the same hypotheses as in Theorem A, every nonzero normal element of $\Omega_G$ is a unit.
\end{thmB}

Theorem B partially answers the open question \cite[Question K]{AB1}.

\subsection{Iwasawa algebras over subgroups of $\SL_2({\mathbb Z}_p)$}
\label{xxsec0.5} Another open question \cite[Question J]{AB1} is,
under hypotheses similar to those in Theorem A, whether there are
any non-trivial prime ideals in $\Omega_G$? This question is
particularly interesting when $G$ an open subgroup of $\SL_2(\Zp)$.
Using \cite[Theorem A]{A} we can prove
\begin{thmC}
Let $G$ be an open torsionfree subgroup of $\SL_2(\Zp)$.
Then every prime ideal in $\Omega_G$ is either zero
or maximal.
\end{thmC}

The proof of Theorem C is independent of \cite{AWZ}.

\subsection{A key step in proof}
\label{xxsec0.7} 
To prove the theorems above we have to overcome several
technical difficulties which seem unrelated to the main 
theorems. The proof is divided into several steps and
we only mention one key step: a control theorem for reflexive
ideals.

\begin{thm}[Theorem \ref{Main}]
Let $(A,A_1)$ be a Frobenius pair satisfying the derivation 
hypothesis, such that $\gr A$ and $\gr A_1$ are UFDs. Let $I$ be a 
reflexive two-sided ideal of $A$. Then $I\cap A_1$ is a 
reflexive two-sided ideal of $A_1$ and $I$ is controlled 
by $A_1$:
\[I = (I\cap A_1)\cdot A.\]
\end{thm}

All undefined terms will be explained later. As an example we 
may take $(A,A_1)$ to be $(\Omega_G,\Omega_{G^p})$. We will
verify that the derivation hypothesis holds for certain groups $G$, and 
the main theorems then follow from the control theorem and induction. 
This control theorem is in fact the heart of the paper, 
on which all our main results are dependent. The control
theorem should be useful for studying Iwasawa algebras
over other classes of groups, such as the nilpotent or solvable groups.

\subsection{A field extension}
\label{FieldExt}
An algebra $A$ over a field is called {\it just
infinite-dimensional} if it is infinite-dimensional and every
nonzero ideal in $A$ is finite codimensional. This is analogous to
the notion of just infinite groups, also known as the almost simple
groups. Theorem C assures us of a large class of just infinite-dimensional algebras with good homological properties.

Several researchers are interested in just infinite-dimensional
algebras over an algebraically closed field (or an infinite field in
general) \cite{BFP,FS}. For ring-theoretic considerations we
introduce another algebra closely related to $\Omega_G$. Let $\FFp$ be a
field of characteristic $p$ (in particular, $\FFp$ could be the
algebraic closure of $\Fp$). Define
$$\FFp G := \FFp[[G]] : =\lim_{\longleftarrow} \FFp [G/N],$$
where the inverse limit is taken over the open normal subgroups $N$
of $G$. This algebra can be obtained by taking a completion of the
algebra $\Omega_{G}\otimes_{{\mathbb F}_p} \FFp$ with respect to the
filtration $\{\fm^n\otimes_{{\mathbb F}_p} \FFp \;|\; n\geq 0\}$
where $\fm$ is the Jacobson radical of $\Omega_{G}$. Under the same
hypotheses, Theorems A, B and C hold for $\FFp G$.

\section{Preliminaries}
\label{Prelim}\subsection{Fractional ideals}\label{Frac}. Let $R$ be
a noetherian domain. It is well-known that $R$ has a skewfield of
fractions $Q$. Recall that a right $R$-submodule $I$ of $Q$ is said
to be a \emph{fractional right $R$-ideal} if $I$ is nonzero and $I
\subseteq uR$ for some nonzero $u\in Q$. When the ring $R$ is
understood, we simply say that $I$ is a fractional right ideal.
Fractional left $R$-ideals are defined similarly. If $I$ is a
fractional right ideal, then
\[ I^{-1} := \{q\in Q : qI \subseteq R\}\]
is a fractional left ideal and there is a similar definition
of $I^{-1}$ for fractional left ideals $I$. Let $I^\ast := \Hom_R(I,R)$.
This is a left $R$-module and there is a natural isomorphism
$\eta_I : I^{-1} \to I^\ast$ that sends $q\in I^{-1}$ to the
right $R$-module homomorphism induced by left multiplication
by $q$.

\subsection{Reflexive right ideals}
\label{ReflIdeals}
Let $I$ be a fractional right ideal and let
$\overline{I} := (I^{-1})^{-1}$ be the \emph{reflexive closure}
of $I$. This is also a fractional right ideal which contains $I$.
Recall that $I$ is said to be \emph{reflexive} if $I = \overline{I}$,
or equivalently if the canonical map $I \to I^{\ast\ast}$ is
an isomorphism.

\begin{prop}
Let $R \hookrightarrow S$ be a ring extension such that
$R$ is noetherian and $S$ is flat as a left and right $R$-module.
Then there is a natural isomorphism
\[\psi_M^i : S\otimes_R\Ext_R^i(M,R) \stackrel{\cong}{\longrightarrow} \Ext_S^i(M\otimes_RS,S)\]
for all finitely generated right $R$-modules $M$ and all $i\geq 0$.
A similar statement holds for left $R$-modules. If in addition $S$
is a noetherian domain, then
\begin{enumerate}[{(}a{)}]
\item $\overline{J\cdot S} = \overline{J}\cdot S$ for all right ideals $J$ of $R$,
\item if $I$ is a reflexive right ideal of $S$, then $I \cap R$
is a reflexive right ideal of $R$.
\end{enumerate}
\end{prop}
\begin{proof} Let $M$ be a finitely generated right $R$-module and define
\[\psi_M : S \otimes_R \Hom_R(M,R) \to \Hom_S(M\otimes_RS,S)\]
by the rule $\psi_M(s\otimes f)(m\otimes t) = sf(m)t$ for all
$s\in S, f\in\Hom_R(M,R),m\in M$ and $t\in R$. This gives
a natural transformation
\[\psi : S \otimes_R \Hom_R(-,R) \to \Hom_S(-\otimes_RS,S)\]
such that $\psi_{R^n}$ is an isomorphism for all $n\geq 0$.
Now let $P_\bullet \to M \to 0$ be a projective resolution of
$M$ consisting of finitely generated free $R$-modules. Using
the flatness assumptions on $S$, we see that
\[\begin{array}{rll}
\Ext_S^i(M\otimes_RS,S) = H^i(\Hom_S(P_\bullet\otimes_RS,S))
& \cong& H^i(S\otimes_R\Hom_R(P_\bullet,R)) = \\
= S\otimes_R H^i(\Hom_R(P_\bullet, R)) &=& S\otimes_R\Ext_R^i(M,R),
\end{array}\]
for all $i$, as required.

(a) The division ring of fractions $Q$ of $R$ embeds naturally
into the division ring of fractions of $S$. Let $I$ be a fractional
right $R$-ideal, so that $I \subseteq uR$ for some $u\in Q\backslash 0$.
Then $IS \subseteq uS$, so $IS$ is a fractional right $S$-ideal.
Now $I^{-1}$ is a fractional left $R$-ideal and $I^{-1}I \subseteq R$, so
\[(SI^{-1})(IS) \subseteq SRS \subseteq S\]
and hence $SI^{-1} \subseteq (IS)^{-1}$. Consider the following
diagram of left $S$-modules:
\[
\xymatrix{
S\otimes_R I^{-1} \ar[r]^{\alpha} \ar[d]_{1\otimes \eta_I} & SI^{-1} \ar[r]^{\iota} & (IS)^{-1} \ar[d]^{\eta_{IS}} \\
S\otimes_RI^{\ast} \ar[r]_{\psi_I} & (I\otimes_RS)^\ast & (IS)^{\ast}. \ar[l]^{\beta}\\
}
\]
Here $\iota$ denotes the inclusion of $SI^{-1}$ into $(IS)^{-1}$ and
$\alpha$ and $\beta$ are the obvious maps. A straightforward check
shows that this diagram commutes. By the remarks made in
\S\ref{Frac} the maps $\eta_I$ and $\eta_{IS}$ are isomorphisms.
Since $S$ is a flat left $R$-module, $\alpha$ is an isomorphism and
similarly $\beta$ is an isomorphism. Now $\psi_I$ is an isomorphism
by the first part, so $\iota$ must also be an isomorphism. We deduce
that $SI^{-1} = (IS)^{-1}$ for all fractional right $R$-ideals $I$.
By symmetry, $I^{-1}S = (SI)^{-1}$ for all fractional left
$R$-ideals $I$.

We may assume that $J$ is nonzero, so that $J$ is a fractional right
ideal, and hence
\[\overline{J\cdot S} = ((JS)^{-1})^{-1} = (SJ^{-1})^{-1} =
(J^{-1})^{-1}S = \overline{J}\cdot S\]
as required.

(b) Again, we may assume that $I\cap R$ is a nonzero, so that $I\cap R$
is a fractional right ideal. Clearly $I\cap R \subseteq \overline{I\cap R}$.
Using part (a) we have
\[ \overline{I\cap R} \subseteq \overline{(I\cap R)\cdot S} \subseteq \overline{I} = I,\]
but $\overline{I\cap R}\subseteq \overline{R} = R$ and hence
$\overline{I\cap R} \subseteq I\cap R$. The result follows.
\end{proof}

\subsection{Pseudo-null modules}
\label{PseudoNull} Let $R$ be an arbitrary ring and $M$ be an
$R$-module. We denote $\Ext_R^j(M,R)$ by $E^j(M)$ . Recall
\cite[Lemma 2.1 and Definition 2.5]{CSS} that an $R$-module $M$ is
said to be \emph{pseudo-null} if $E^0(N) = E^1(N) = 0$ for any
submodule $N$ of $M$. Part (b) of the Proposition below shows that
this extends the notion of pseudo-zero modules in the sense of
\cite[Chapter VII, \S 4.4, Definition 2]{Bour}.

\begin{lem}
Let $0 \to X \to Y \to Z \to 0$ be an exact sequence of $R$-modules.
Then $Y$ is pseudo-null if and only if $X$ and $Z$ are pseudo-null.
\end{lem}
\begin{proof} This appears in \cite[\S 2]{CSS} and follows easily
from the long exact sequence of cohomology.
\end{proof}

The following alternative characterisation of pseudo-null modules
over noetherian domains is well-known, but we include a proof for
the convenience of the reader.

\begin{prop} Let $R$ be a noetherian domain and let $M$ be a
finitely generated $R$-module.
\begin{enumerate}[{(}a{)}]
\item $M$ is pseudo-null if and only if $\ann(x)^{-1} = R$ for all $x\in M$.
\item If $R$ is commutative then $M$ is pseudo-null if and only if $\Ann_R(M)^{-1} = R$.
\end{enumerate}
\end{prop}
\begin{proof} (a) Suppose $M$ is pseudo-null and let $x\in M$.
The short exact sequence $0 \to \ann(x) \to R \to xR \to 0$
induces the long exact sequence
\[0 \to E^0(xR) \to E^0(R) \to E^0(\ann(x)) \to E^1(xR) \to 0\]
and $E^0(xR) = E^1(xR) = 0$ since $M$ is pseudo-null. Hence
$\ann(x)^{-1} = R^{-1} = R$ by the remarks made in $\S\ref{Frac}$.

Conversely, suppose that $\ann(x)^{-1} = R$ for all $x \in M$. It
will be enough to show that $E^0(M) = E^1(M) = 0$. Let $N = yR$ be a
quotient of a cyclic submodule $xR$ of $M$. Then $\ann(x) \subseteq
\ann(y)$, so $R \subseteq \ann(y)^{-1} \subseteq \ann(x)^{-1} = R$.
Hence $\ann(y)^{-1} = R$ and the above long exact sequence shows
that $E^0(N) = E^1(N) = 0$.

Because $M$ is finitely generated, $M$ is an extension of finitely
many modules $M_1, \ldots, M_k$ such that each $M_i$ is isomorphic
to a quotient of a cyclic submodule of $M$. The result now follows
from a long exact sequence.

(b) Suppose that $\Ann_R(M)^{-1} = R$. Since $\Ann_R(M) \subseteq
\ann(x)$ for all $x\in M$, part (a) implies that $M$ must be
pseudo-null.

Conversely, suppose that $M$ is pseudo-null and let $x_1,\ldots, x_k$
be a generating set for $M$. Since $M$ is pseudo-null, $\ann(x_i)^{-1} = R$
for all $i$. Since $R$ is commutative, $\Ann_R(M)$ contains the
product $\ann(x_1)\cdots\ann(x_k)$ and it follows easily that
$\Ann_R(M)^{-1} = R$.
\end{proof}

\subsection{Unique factorisation domains}
\label{UFD}

\begin{lem} Let $R$ be a commutative noetherian unique factorisation
domain (UFD) and $I$ be a nonzero ideal of $R$. Then $\overline{I} =
xR$ for some $x \in R$ and $xR/I$ is pseudo-null. Moreover, if $R$
is a graded ring and $I$ is a graded ideal, then $x$ is homogeneous.
\end{lem}

\begin{proof}
By \cite[Chapter VII, \S 4.2, Example 2 and \S 3.1, Definition 1]{Bour},
every reflexive ideal of $R$ is necessarily principal. Hence
$\overline{I} = xR$ for some $x \in R$.

Now let $J = \Ann_R(xR/I) = x^{-1}I$ and suppose that $q \in
J^{-1}$. Then $qJ = qx^{-1}I \subseteq R$ and so $qx^{-1} \in I^{-1}
= \overline{I}^{-1} = x^{-1}R$. Therefore $q\in R$ and $J^{-1} = R$.
Hence $xR/I$ is pseudo-null by Proposition \ref{PseudoNull} (b).

Suppose finally that $R$ and $I$ are graded. Then we can find
a nonzero homogeneous element $y\in I$. Since $I \subseteq xR$
we see that $x$ is a factor of $y$. Because $R$ is a domain,
homogeneous elements can only have homogeneous factors, so
$x$ is necessarily homogeneous.
\end{proof}

\subsection{Filtered rings}
\label{FiltBasics}
A \emph{filtered ring} is a ring $R$ with a filtration $FR=\{F_nR :
n\in\mathbb{Z}\}$ consisting of additive subgroups of $R$ such that
$R = \bigcup_{n\in\mathbb{Z}}F_nR, \; 1 \in F_0R, \;
F_nR \subseteq F_{n+1}R$ and $F_nRF_mR \subseteq F_{n+m}R$ for
all $n,m\in\mathbb{Z}$. Our filtrations will always be
\emph{separated}, meaning that $\bigcap_{n\in\mathbb{Z}}F_nR = 0$.
If $x$ is a nonzero element of $R$, there exists a unique
$n\in\mathbb{Z}$, which is called the \emph{degree} of $x$
and written $n = \deg x$, such that $x \in F_nR - F_{n-1}R$.

The abelian group $\gr R := \oplus_{n\in\mathbb{Z}}F_nR/F_{n-1}R$
becomes a graded ring with multiplication induced by that of $R$
and is called the \emph{associated graded ring} of $R$ with respect
to $FR$. The \emph{principal symbol} of a nonzero element $x$ of
$R$ of degree $n$ is
\[\gr x:= x + F_{n-1}R \in F_nR / F_{n-1}R \subseteq \gr R.\]
If $\gr R$ is a domain then $\gr(xy) = \gr(x)\gr(y)$ for any
nonzero $x,y\in R$.

The \emph{Rees ring} of $R$ (with respect to the filtration
$FR$) is the following subring of the Laurent polynomial
ring $R[t,t^{-1}]$:
\[\widetilde{R} := \bigoplus_{n \in \mathbb{Z}} t^n F_n R.\]
The Rees ring comes equipped with two natural surjective
ring homomorphisms $\pi_1 : \widetilde{R} \to R$ and
$\pi_2 : \widetilde{R} \to \gr R$ which send the indeterminate
$t$ to one and zero, respectively. The map $\pi_1$ is sometimes
called \emph{dehomogenisation}.
\section{Frobenius Pairs}
\label{xxsec2}
\subsection{The classical Frobenius map}
\label{xxsec2.1}
Let $K$ be a field of characteristic $p$ and let $B$ be a
commutative $K$-algebra. Then the Frobenius map $x \mapsto x^p$ is a ring
endomorphism of $B$ and gives an isomorphism of $B$ onto its image
\[B^{[p]} := \{b^p : b \in B\} \]
in $B$ provided $B$ is reduced. We remark at this point that any
derivation $d : B \to B$ is $B^{[p]}$-linear:
\[d(a^pb) = a^pd(b) + pa^{p-1}d(b) = a^pd(b)\]
for all $a,b \in B$.

\subsection{Frobenius pairs}
\label{FrobPair} Let $t$ be a positive integer.
Whenever $\{y_1,\ldots,y_t\}$ is a $t$-tuple of elements of $B$ and
$\alpha = (\alpha_1,\ldots,\alpha_t)$ is a $t$-tuple of nonnegative
integers, we define
\[{\mathbf{y}}^{\alpha} =y_1^{\alpha_1}\cdots y_t^{\alpha_t}.\]
Let $[p-1]$ denote the set $\{0,1,\ldots,p-1\}$ and let $[p-1]^t$ be
the product of $t$ copies of $[p-1]$.

\begin{defn}
Let $A$ be a complete filtered $\FFp$-algebra and let $A_1$ be a
subalgebra of $A$. We always view $A_1$ as a filtered subalgebra of
$A$, equipped with the subspace filtration $F_nA_1 := F_nA \cap
A_1$. We say that $(A,A_1)$ is a \emph{Frobenius pair} if the
following axioms are satisfied:
\begin{enumerate}[{(} i {)}]
\item $A_1$ is closed in $A$,
\item $\gr A$ is a commutative noetherian domain, and we write $B=\gr A$,
\item the image $B_1$ of $\gr A_1$ in $B$ satisfies $B^{[p]} \subseteq B_1$, and
\item there exist homogeneous elements $y_1,\ldots,y_t \in B$ such that
\[B = \bigoplus_{\alpha \in [p-1]^t}  B_1 \mathbf{y}^{\alpha}.\]
\end{enumerate}
\end{defn}

\begin{rem} It is easy to see that $A_1$ is closed in $A$ if and
only if the subspace filtration $\{F_nA_1\}_{n\in\mathbb{Z}}$
on $A_1$ is complete.
\end{rem}

The canonical example to keep in mind is given by Iwasawa algebras
of uniform pro-$p$ groups $G$. We will show in $\S\ref{IwaFrobPair}$
that $(KG,KG^p)$ is always a Frobenius pair.

We will now deduce some consequences of the axioms.

\subsection{The structure of $A$ as an $A_1$-module}
\label{FreeMod}
Let $(A,A_1)$ be a Frobenius pair. We can view $A$ as an $A_1$-bimodule.
Let us choose elements $u_1,\ldots,u_t\in A$
such that $\gr u_i = y_i$ for all $i$ and set
$\mathbf{u}^\alpha := u_1^{\alpha_1}\cdots u_t^{\alpha}$
for all $\alpha\in\mathbb{N}^t$.

\begin{lem} The $A$ is a free left and right $A_1$-module
with basis $\{\mathbf{u}^\alpha : \alpha \in [p-1]^t\}:$
\[A = \bigoplus_{\alpha\in[p-1]^t} A_1 \cdot \mathbf{u}^\alpha
= \bigoplus_{\alpha\in[p-1]^t} \mathbf{u}^\alpha \cdot A_1.\]
\end{lem}

\begin{proof}
By symmetry it is sufficient to prove the statement about
left modules, say. Suppose for a contradiction that
$\sum_{\alpha\in T} a_{\alpha}\mathbf{u}^\alpha = 0$,
where $\{a_\alpha \in A_1 : \alpha \in T\}$
is some collection of nonzero elements and $T \subseteq [p-1]^t$
is a nonempty indexing set. Let $n$ denote the maximum of
the degrees of the $a_{\alpha}\mathbf{u}^\alpha$ and
let $S$ denote the subset of $T$ consisting of those
indices $\alpha$ where this maximum is attained. Then
\[\left(\sum_{\alpha\in T}
a_\alpha \mathbf{u}^\alpha\right) + F_{n-1}A =
\sum_{\alpha\in S} \gr a_\alpha \cdot \mathbf{y}^\alpha = 0,\]
which is contradictory to Definition \ref{FrobPair}(iv). Thus the
sum $M := \sum_{\alpha\in[p-1]^t} A_1\cdot \mathbf{u}^\alpha$
is direct.

Now $M$ is a filtered $A_1$-submodule of $A$ and $\gr M$
coincides with $\gr A$. Since $A_1$ is complete, $M$
is equal to $A$ and the result follows.
\end{proof}

\subsection{Derivations}
\label{Derivations}
Let $B_1 \subseteq B$ be commutative rings of
characteristic $p$, such that $B^{[p]} \subseteq B_1$ and
\[B = \bigoplus_{\alpha\in[p-1]^t} B_1\mathbf{y}^\alpha\]
for some elements $y_1,\ldots, y_t$ of $B$.

Fix $j = 1,\ldots,t$ and let $\epsilon_j$ denote the $t$-tuple of
integers having a $1$ in the $j$-th position and zeros elsewhere.
We define a $B_1$-linear map $\partial_j : B \to B$ by setting
\[\partial_j\left(\sum_{\alpha\in[p-1]^t} u_\alpha
\mathbf{y}^\alpha \right) := \sum_{\stackrel{\alpha\in[p-1]^t}
{\alpha_j > 0}} \alpha_ju_\alpha \mathbf{y}^{\alpha - \epsilon_j}.\]

Let $\mathcal{D} := \Der_{B_1}(B)$ denote the set of all $B_1$-linear
derivations of $B$. We now collect some very useful results about
$\mathcal{D}$ and its natural action on $B$. In particular, we can
give a complete characterisation of the $\mathcal{D}$-stable ideals of $B$.

\begin{prop}
\begin{enumerate}[{(}a{)}]
\item The map $\partial_j$ is a $B_1$-linear derivation of $B$ for each $j$.
\item $\mathcal{D} = \bigoplus_{j=1}^t B \partial_j.$
\item For any $x \in B$, $\mathcal{D}(x) = 0$ if and only if $x \in B_1$.
\item An ideal $I \subseteq B$ is $\mathcal{D}$-stable if and only if it is controlled by $B_1$:
\[I = (I\cap B_1)B.\]
\end{enumerate}
\end{prop}
\begin{proof} (a) Because the $y_i$'s generate $B$ as a $B_1$-algebra
and $\partial_j$ is $B_1$-linear by definition, to show that
$\partial_j$ is a derivation it is sufficient to check that
\[\partial_j(\mathbf{y}^\alpha\cdot y_i) = \partial_j(\mathbf{y}^\alpha)y_i
+ \mathbf{y}^\alpha\cdot \partial_j(y_i)\]
for all $\alpha\in[p-1]^t$ and all $i=1,\ldots,t$.
This can be easily verified, using the fact that $y_k^p \in B_1$
for all $k=1,\ldots, t$.

(b) If $b_j\in B$ are such that $\sum_{j=1}^tb_j\partial_j = 0$, then
$b_i = (\sum_{j=1}^tb_j\partial_j)(y_i) = 0$ for all $i$, so the sum
above is direct. Finally, if $f\in\mathcal{D}$, then it is easy to
see that $f$ and $\sum_{j=1}^t f(y_j)\partial_j$ agree on every
element of $B$ with the form $u\cdot\mathbf{y}^\alpha$ for $u\in
B_1$, so $f = \sum_{j=1}^t f(y_j)\partial_j$ and the result follows.

(c) Suppose $x \notin B_1$ and write $x = \sum_{\alpha \in [p-1]^t}
x_\alpha \mathbf{y}^\alpha$. Then $x_\alpha \neq 0$ for some $\alpha
\neq 0$ and so $\alpha_j\neq 0$ for some $j$. Hence $\partial_j(x)
\neq 0$. The converse is trivial.

(d) ($\Leftarrow$) Let $J = I\cap B_1$. For any $f \in \mathcal{D}$ we have
\[f(I) = f(JB) = Jf(B) \subseteq JB = I\]
so $I$ is $\mathcal{D}$-stable.

($\Rightarrow$) Let $I$ be a $\mathcal{D}$-stable ideal and let $J =
I\cap B_1$. Note that the extension $B_1/J \subseteq B/JB$ satisfies
the same conditions as $B_1 \subseteq B$, and the image of $I$ in
$B/JB$ is stable under every $B_1/J$-linear derivation of $B/JB$ by
part (b). Without loss of generality we may therefore assume that
$I\cap B_1 = 0$, and it will be enough to show that $I = 0$.

Suppose for a contradiction that $I \neq 0$. If
$u = \sum_{\alpha\in[p-1]^t}u_\alpha\mathbf{y}^\alpha \in B$
is a nonzero element, define
\[m(u) := \max\{\alpha_1 + \cdots + \alpha_t: u_\alpha \neq 0\}\]
and choose $u \in I \backslash 0$ such that $m(u)$ is minimal.
If $\partial_j(u) \neq 0$ for some $j$ then $m(\partial_j(u)) < m(u)$
and $\partial_j(u) \in I\backslash 0$ contradicting the minimality of $m(u)$.
Hence $\partial_j(u) = 0$ for all $j$ and therefore $u\in B_1$ by
parts (b) and (c). But then $I\cap B_1 \neq 0$, a contradiction.
Hence $I = 0$ as required.
\end{proof}

\section{A control theorem for normal elements}
\subsection{Main result}\label{ContNormal}
The purpose of this section is to prove the following
\begin{thm}
Let $(A,A_1)$ be a Frobenius pair satisfying the derivation 
hypothesis, suppose that
$B_1$ is a UFD and let $w \in A$ be a normal element. Then the
two-sided ideal $wA$ of $A$ is controlled by $A_1$:
\[wA = (wA \cap A_1)\cdot A.\]
\end{thm}

The derivation hypothesis is explained below in \S\ref{Star} and the
proof is given in \S\ref{ProofContNormal}.

\subsection{Inducing derivations on $\gr A$}
\label{Source}
Let $A$ be a filtered ring with associated graded ring $B$ and
let $a\in A$. Suppose that there is an integer $n\geq 0$ such that
\[[a,F_kA] \subseteq F_{k-n}A\]
for all $k\in\mathbb{Z}$. This induces linear maps
\[\begin{array}{cccc}
\{a, - \}_n :& \frac{F_kA}{F_{k-1}A} &\to & \frac{F_{k-n}A}{F_{k-n-1}A} \\
\quad &\quad \\
& b + F_{k-1}A &\mapsto &[a,b] + F_{k-n-1}A
\end{array}
\]
for each $k\in\mathbb{Z}$ which piece together to give a
graded derivation
\[\{a, - \}_n : B \to B.\]

The idea of inducing derivations of $\gr A$ in this way was
first suggested to the first author by Chris Brookes and
later independently by Ken Brown.

\begin{defn}
A \emph{source of derivations} for a Frobenius pair $(A,A_1)$ is
a subset $\mathbf{a} = \{a_0,a_1,a_2,\ldots\}$ of $A$ such that
there exist functions $\theta, \theta_1:\mathbf{a} \to\mathbb{N}$
satisfying the following conditions:
\begin{enumerate}[{(} i {)}]
\item
$[a_r, F_kA] \subseteq F_{k - \theta(a_r)}A$
for all $r\geq 0$ and all $k\in\mathbb{Z}$
\item
$[a_r, F_kA_1] \subseteq F_{k - \theta_1(a_r)}A$ for all $r\geq 0$
and all $k\in\mathbb{Z}$,
\item $\theta_1(a_r) - \theta(a_r) \to \infty$ as $r \to \infty$.
\end{enumerate}
Let $\mathcal{S}(A,A_1)$ denote the set of all sources of
derivations for $(A,A_1)$.
\end{defn}

The reason behind this definition will hopefully become clear after
Proposition \ref{xClosures} below. By (i), any source of derivations
$\mathbf{a}$ generates a sequence of graded derivations
$\{a_r,-\}_{\theta(a_r)}$ of $B = \gr A$. These derivations are
$B_1$-linear for sufficiently large $r$ by parts (ii, iii).

The subset $\{0\}$ is clearly an example of a source of derivations.
Somewhat less trivially, we will show in Corollary \ref{SourcesIwa}
that if $G$ is a uniform pro-$p$ group and $g \in G$, then
$\{g, g^p, g^{p^2}, \ldots\}$ is a source of derivations for
the the Frobenius pair $(KG,KG^p)$.

\subsection{The delta function}
\label{DeltaFunction} Let $(A,A_1)$ be a Frobenius pair and $n$ be
an integer. Each filtered part $F_nA_1$ is closed in $A_1$ by
definition of the filtration topology, and $A_1$ is closed in $A$ by
assumption. Hence $F_nA_1$ is closed in $A$, which can be expressed
as follows:
\[F_nA_1 = \bigcap_{k\geq 0} \left(F_nA_1 + F_{n-k}A\right).\]

We can now define a key invariant of elements of $A$.

\begin{defn}
For any $w \in A$, let $n = \deg w$. Define
\[\delta(w) = \left\{ \begin{array}{cl}
\max\{k : w \in F_nA_1 + F_{n-k}A\} \quad &\mbox{ if } \quad w\notin A_1 \\
\infty \quad &\mbox{ if } \quad w \in A_1.
\end{array} \right. \]
\end{defn}

Clearly $\delta(w)\geq 0$. Note that if $w\in F_nA \backslash A_1$,
then $w \notin F_nA_1 + F_{n-k}A$ for some $k\geq 0$ by the above
remarks, so the definition makes sense and $\delta(w)$ is finite.
The number $\delta(w)$ measures how closely the element $w$ can be
approximated by elements of $A_1$. It should be remarked that
$\delta(w) > 0$ if and only if $\gr w \in B_1$, since both
conditions are equivalent to $w \in F_nA_1 + F_{n-1}A$.

Now suppose that $w\in A \backslash A_1$. By the definition of $\delta$,
we can find elements $x\in F_nA_1$ and $y\in F_{n-\delta}A$ such
that $w = x + y$; if
$\delta = 0$ we take $x$ to be zero. Note that $y \notin F_{n - \delta - 1}A$
by the maximality of $\delta$ and hence
\[Y_w := \gr y = y + F_{n - \delta - 1}A.
\]
In view of our assumption on $x$, we have $Y_w = \gr w$ when $\delta
= 0$.

\subsection{$\mathbf{a}$-closures}
\label{xClosures}
If $w$ is an element of a right ideal $I$ of $A$,
then the symbol of $w$, $\gr w$ always lies in the associated graded
ideal $\gr I$ of $B$. Naturally there are many elements $w$ having
the same symbol, so some information is lost when one passes to the
symbol of $w$. It turns out that if the ideal $I$ is two-sided,
there is a way to save some of this information.

\begin{defn} Let $\mathbf{a}$ be a source of derivations for a
Frobenius pair $(A,A_1)$ and $I$ be a graded ideal of $B$. We say
that the homogeneous element $Y$ of $B$ lies in the
\emph{$\mathbf{a}$-closure} of $I$ if $\{a_r,Y\}_{\theta(a_r)}$ lies
in $I$ for all $r \gg 0$.
\end{defn}

\begin{prop} Let $(A,A_1)$ be a Frobenius pair, $I$ be a
two-sided ideal of $A$ and $w\in I\backslash A_1$. Then $Y_w$ lies
in the $\mathbf{a}$-closure of $\gr I$ for any source of derivations
$\mathbf{a}$.
\end{prop}
\begin{proof} Let us write $w = x + y$ as in the previous subsection.
Since $\mathbf{a}$ is a source of derivations, we can find an
integer $r_0\geq 1$ such that $\theta_1(a_r) - \theta(a_r) > \delta$
for all $r \geq r_0$. Therefore
\[\begin{aligned}
{[a_r,x]}\; &\in F_{n - \theta_1(a_r)}A \subseteq
F_{n - \delta - \theta(a_r) - 1}A \quad\mbox{and} \\
[a_r,y] \; &\in F_{n - \delta - \theta(a_r)}A,
\end{aligned}
\]
for all $r \geq r_0$. Hence
$$\begin{aligned}
{[a_r,w]} \; &\in F_{n - \delta - \theta(a_r)}A,
\quad \mbox{and} \\
[a_r,w] \; & \equiv [a_r,y] \mod
F_{n - \delta - \theta(a_r) - 1}A
\end{aligned}
$$
for all $r\geq r_0$. We can rewrite the above as follows:
\[[a_r,w] + F_{n-\delta-\theta(a_r) - 1}A =
[a_r,y] + F_{n-\delta-\theta(a_r) - 1}A
= \{a_r, Y_w\}_{\theta(a_r)}\]
for $r \geq r_0$. Since $w \in I$ and $I$ is a two-sided ideal,
this element must always lie in the ideal $\gr I$ of $B$, and
hence $Y_w$ lies in the $\mathbf{a}$-closure of $\gr I$ as required.
\end{proof}
Each source of derivations $\mathbf{a}$ gives rise to a sequence of
derivations $\{a_r,-\}_{\theta(a_r)}$ of $B$, and some or all of
these could well be zero. To ensure that we get an interesting
supply of derivations of $B$, we now introduce a condition which
holds for Iwasawa algebras of only rather special uniform pro-$p$ groups.

\subsection{Derivation hypothesis}
\label{Star}
Recall that $\mathcal{D}$ denotes the set of all $B_1$-linear
derivations of $B$ and $\mathcal{S}(A,A_1)$ denotes the set of
all sources of derivations for $(A,A_1)$. Our derivation hypothesis
is really concerned with the action of the derivations induced by
$\mathcal{S}(A,A_1)$ on the graded ring $B$.

\begin{defn}
Let $(A,A_1)$ be a Frobenius pair and $X \in B$ be an arbitrary
homogeneous element. We say that $(A,A_1)$ satisfies the {\sf derivation 
hypothesis} if whenever a homogeneous element $Y \in B$ lies in the
$\mathbf{a}$-closure of $XB$ for all
$\mathbf{a}\in\mathcal{S}(A,A_1)$, we must have $\mathcal{D}(Y)
\subseteq XB$.
\end{defn}

Assuming the derivation hypothesis, it is possible to ``clean" a normal
element by multiplying it by a unit. The following Proposition
forms the inductive step in the proof of Theorem \ref{ContNormal}.

\begin{prop}
Let $(A,A_1)$ be a Frobenius pair satisfying the derivation hypothesis and
let $w\in A \backslash A_1$ be a normal element.
Then there exists a unit $u\in A$ such that $\delta(wu) > \delta(w)$.
Moreover, if $\delta(w) > 0$ then $u = 1 - c$ for some $c\in F_{-\delta(w)}A$.
\end{prop}
\begin{proof}
Write $w = x + y$ as in $\S\ref{DeltaFunction}$, let $X = \gr w$
and $Y = Y_w = \gr y$. By Proposition \ref{xClosures}, $Y$ lies
in the $\mathbf{a}$-closure of $\gr wA = XB$ for all
$\mathbf{a}\in \mathcal{S}(A,A_1)$ and hence $\mathcal{D}(Y) \subseteq XB$
because $(A,A_1)$ satisfies the derivation hypothesis.

Suppose first that $\delta := \delta(w) = 0$, so that $Y = X$. Then
the ideal $XB$ of $B$ is $\mathcal{D}$-stable and is hence
controlled by $B_1$ by Proposition \ref{Derivations}(d):
\[XB = (XB \cap B_1)\cdot B.\]
Because $B$ is a free $B_1$-module, $XB \cap B_1$ is a reflexive
ideal of $B_1$ by Proposition \ref{ReflIdeals}(b). Since $B_1$
is a UFD by assumption, $XB \cap B_1 = X_1B_1$ for some homogeneous
element $X_1 \in B_1$ by Lemma \ref{UFD}.

Hence $XB = X_1B$ and we can therefore find a homogeneous unit $U\in
B$ such that $X_1 = XU$. Choose $u,v\in A$ such that $\gr u = U$ and
$\gr v = U^{-1}$; then $uv \equiv 1 \mod F_{-1}A$. But $A$ is
complete so $1 + F_{-1}A$ consists of units in $A$ and hence $u$ is
a unit. Since $\gr(wu) = XU = X_1 \in B_1$, it follows that
$\delta(wu) > 0 = \delta(w)$ as required.

Now suppose that $\delta > 0$; then $X$ must lie in $B_1$. Applying
Proposition \ref{Derivations}(c) to the image of $Y$ in $B/XB$
yields that
\[Y \in XB + B_1.\]
Since $X$ and $Y$ are homogeneous, we can find homogeneous
elements $C \in B$ and $Z \in B_1$ such that
\[Y = XC + Z;\]
moreover $\deg Y = \deg XC$ if $XC \neq 0$ and $\deg Y = \deg Z$ if $Z \neq 0$.

Suppose for a contradiction that $C = 0$. Then $Y = Z \in B_1$. Hence
we can find $x' \in F_{n-\delta}A_1$ such that
\[x' \equiv y \mod F_{n-\delta-1}A.\]
Thus $w - (x + x') \in F_{n - \delta - 1}A$, which is
contradictory to the maximality of $\delta$. So $C \neq 0$ and hence
$\deg C = \deg Y - \deg X = - \delta$. Note that $\deg C < 0$.

We can find $c \in A$ such that $\gr c = C$. Then
\[w(1-c) = (x + y)(1 - c) \equiv x + y - xc \mod F_{n - \delta - 1}A\]
since $\deg(yc) < n -\delta$. But
\[y - xc + F_{n-\delta-1}A = Y - XC = Z \in B_1,\]
so we can find $z \in A_1$ such that $y - xc \equiv z \mod F_{n - \delta - 1}A$ and hence
\[w(1-c) - (x + z) \in F_{n - \delta - 1}A.\]
Since $\deg Z = \deg Y$ if $Z \neq 0$, $z \in F_{n - \delta}A$ and
hence $x + z \in F_nA_1$. This implies that $\delta(w(1-c)) > \delta = \delta(w)$.

Finally, since $c \in F_{-\delta(w)}A \subseteq F_{-1}A$ and $A$
is complete, $u := 1 - c$ is a unit in $A$ and $\delta(wu) > \delta(w)$
by construction.
\end{proof}

\subsection{Proof of Theorem \ref{ContNormal}}
\label{ProofContNormal}
\begin{proof}
It will be enough to construct a unit $u \in A$ such that $wu \in A_1$.

If $w$ already happens to lie in $A_1$ then we can take $u = 1$, so
assume that $w \notin A_1$. By Proposition \ref{Star} there exists a
unit $u_0 \in A$ such that $\delta(wu_0) > 0$.

Let $w_0 := wu_0$. Using Proposition \ref{Star} we can inductively
construct a sequence of normal elements $w_1, w_2, \ldots$ of $A$
and a sequence of elements $c_1, c_2, \ldots$ of $A$, such that
for all $i\geq 0$,
\begin{itemize}
\item $c_{i+1} \in F_{-\delta(w_i)}A$,
\item $w_{i+1} = w_i(1 - c_{i+1})$,
\item $\delta(w_{i+1}) > \delta(w_i)$ if $w_i \notin A_1$.
\end{itemize}
Here we interpret $F_{-\infty}A$ as the zero subspace of $A$.
With this convention, the sequence $c_i$ converges to zero as
$i\to\infty$ by construction, so the limit
\[u := \lim_{i\to \infty} u_0(1 - c_1)\cdots(1 - c_i)\]
exists in $A$ by the completeness of $A$. Note that $u$ is
unit because we can write down an inverse having the same
form as $u$, and that $wu = \lim_{i\to\infty}w_i$.

We will now show that $wu$ lies in $A_1$. Since $A_1$ is closed in
$A$, it will be sufficient to show that $wu \in A_1 + F_kA$ for all
$k\in\mathbb{Z}$. Let $n = \deg w_0$ and note that $\deg w_i = n$
for all $i\geq 0$ by construction. Since $w_i \to wu$ and
$\delta(w_i) \to \infty$ as $i\to \infty$, we see that for $i \gg
0$, $wu - w_i \in F_kA$ and $w_i \in F_nA_1 + F_kA$. Hence $wu \in
F_nA_1 + F_kA \subseteq A_1 + F_kA$, as required.
\end{proof}

\section{Microlocalisation}
\label{MicroLoc}

\subsection{Notation}
\label{MicLocBasic} We briefly recall some basic facts about the
theory of algebraic microlocalisation, following \cite{Li} and
\cite{AVV}. Our notation will be slightly non-standard. Throughout
$\S \ref{MicroLoc}$ we will make the following assumptions:
\begin{itemize}
\item $R$ is a filtered ring whose Rees ring $\widetilde{R}$ is noetherian,
\item $T$ is a right Ore subset of $\gr R$ consisting of
homogeneous regular elements.
\end{itemize}
Since $R$ and $\gr R$ are homomorphic images of $\widetilde{R}$ by
$\S \ref{FiltBasics}$, these rings must also be noetherian. We
should remark at this point that if the filtration on $R$ is
complete and $\gr R$ is noetherian, then the filtration on $R$ is
\emph{zariskian}: see \cite[Chapter II, \S 2.1, Definition 1 and \S
2.2, Proposition 1]{LV}. In particular $\widetilde{R}$ is
necessarily noetherian.

\subsection{Lifting Ore sets}
\label{LiftOre}
Let $\widetilde{T}$ denote the homogeneous inverse image
of $T$ in $\widetilde{R}$:
\[\widetilde{T} := \{r \in \widetilde{R} : r \mbox{ is homogeneous and } \pi_2(r) \in T\}.\]
It can be shown that $\widetilde{T}$ is a right Ore set
in $\widetilde{R}$ \cite[Corollary 2.2]{Li}, so we may form the
Ore localisation $\widetilde{R}_{\widetilde{T}}$. This is still
a $\mathbb{Z}$-graded ring.

Let $S:=\pi_1(\widetilde{T}) \subseteq R$. This is a right
Ore set in $R$ and in fact
\[S = \{r \in R : \gr r \in T\}.\]
Note that $S$ consists of regular elements in $R$, since every
element of $T$ is assumed to be regular. It follows that $R$ embeds
into the Ore localisation $R_S$.

The surjection $\pi_1 : \widetilde{R} \twoheadrightarrow R$ extends
to surjection $\pi_1 : \widetilde{R}_{\widetilde{T}} \twoheadrightarrow R_S$.
The grading on $\widetilde{R}_{\widetilde{T}}$ induces a
filtration on $R_S$, as in \cite[Proposition 2.3(1)]{Li}:
\[F_nR_S := \pi_1((\widetilde{R}_{\widetilde{T}})_n).\]
Here $(\widetilde{R}_{\widetilde{T}})_n$ denotes the $n^{\mbox{th}}$-graded
part of $\widetilde{R}_{\widetilde{T}}$.

\begin{lem}
The filtration on $R_S$ is given explicitly by the formula
\[F_nR_S = \{rs^{-1} : r\in R, s\in S\quad\mbox{and}\quad\deg r - \deg s \leq n\}\]
for all integers $n$. This filtration is zariskian.
\end{lem}
\begin{proof} Before we begin the proof, let us observe that if $x\in R$
is nonzero and $y\in S$ then $(\gr x)(\gr y) \neq 0$ (because $\gr x \neq 0$
and $\gr y \in T$ is regular) and hence $\deg(xy) = \deg x + \deg y$.

Let $L_n = \{rs^{-1} : r\in R, s\in S$ and $\deg r - \deg s \leq n\}$.
Decoding the definition of $F_nR_S$, we see that $F_nR_S$ is in fact
the additive subgroup of $R_S$ generated by $L_n$. It will therefore
be sufficient to show that $L_n$ is closed under addition.

So let $r_1s_1^{-1}$ and $r_2s_2^{-1}$ be elements of $L_n$ for
some $r_i\in R$ and $s_i\in S$. We can find $u_1\in S$ and $u_2\in R$
such that $s_1u_1 = s_2u_2 = s$ say; then
$r_1s_1^{-1} + r_2s_2^{-1} = (r_1u_1 + r_2u_2)s^{-1}$.
Since $s_1,s_2\in S$, we have $\deg s = \deg s_1 + \deg u_1 = \deg s_2 + \deg u_2$
by the first paragraph. Now
\[\begin{array}{ll}\deg(r_1u_1+r_2u_2) - \deg s &\leq \max\{\deg r_1 +
\deg u_1,\deg r_2 + \deg u_2 \} - \deg s = \\ &= \max \{\deg r_1 - \deg s_1,
\deg r_2 - \deg s_2\} \leq n
\end{array}\]
so $(r_1u_1+r_2u_2)s^{-1} \in L_n$, as required.

The last assertion follows from \cite[Proposition 2.8]{Li}.
\end{proof}
\subsection{Microlocalisation of rings}
\label{MicRings}
\begin{defn} The \emph{microlocalisation} of $R$ at $T$ is the
completion $Q_T(R)$ of $R_S$ with respect to the filtration on
$R_S$ described in $\S$\ref{LiftOre}.
\end{defn}
We record some useful properties enjoyed by microlocalisation.

\begin{prop}
\begin{enumerate}[{(}a{)}]
\item $Q_T(R)$ is a complete filtered ring,
\item $F_nQ_T(R)$ is the closure of $F_nR_S$ in $Q_T(R)$,
\item $R$ embeds into $Q_T(R)$,
\item $Q_T(R)$ is a flat right $R$-module,
\item there are natural isomorphisms
\[\gr Q_T(R) \cong \gr(R_S) \cong (\gr R)_T.\]
\end{enumerate}
\end{prop}
\begin{proof} Parts (a) and (b) are clear from the definition.
We have seen in $\S\ref{LiftOre}$ that $R$ embeds into $R_S$,
and the filtration on $R_S$ is separated by Lemma \ref{LiftOre}
and \cite[Chapter II, \S 2.1, Theorem 2]{LV}. Hence $R_S$ embeds
into $Q_T(R)$ and part (c) follows. Parts (d) and (e) follow
from \cite[Corollary 3.20(1) and Proposition 3.10]{AVV}.
\end{proof}

\subsection{Microlocalisation of modules} \label{MicMod}
Let $M$ be a finitely generated right $R$-module. We define
the \emph{microlocalisation} of $M$ at $T$ to be
\[Q_T(M) := M\otimes_RQ_T(R).\]
This is naturally a right $Q_T(R)$-module. Recall that a filtration
on $M$ is said to be \emph{good} if the associated Rees module is
finitely generated over $\widetilde{R}$.

\begin{lem} Let $M$ be a finitely generated $R$-module equipped
with some good filtration, and $N$ be a submodule of $M$. Then
\begin{enumerate}[{(}a{)}]
\item $\gr Q_T(M) \cong (\gr M)_T,$
\item the Ore localisation $M_S$ is a dense $R_S$-submodule of $Q_T(M)$,
\item $Q_T(N)$ can be identified with a $Q_T(R)$-submodule of $Q_T(M)$,
\item the tensor filtration on $Q_T(N)$ coincides with the
subspace filtration induced from $Q_T(M)$,
\item if $L$ is another submodule of $M$, then 
$Q_T(N) \cap Q_T(L) = Q_T(N\cap L).$
\end{enumerate}
\end{lem}
\begin{proof}
We should remark that the filtration on $Q_T(M) = M \otimes_R Q_T(R)$
is the \emph{tensor filtration} in the sense of \cite[Chapter I, \S 6]{LV}.
Part (a) follows from \cite[Proposition 3.10 and Corollary 3.20(2)]{AVV}
and part (b) follows from \cite[Corollary 2.5(3)]{Li}, whereas parts
(c) and (d) follow from \cite[Corollary 3.16(3)]{AVV}.

Finally, $Q_T(R)$ is a flat $R$-module by Proposition \ref{MicRings}(d),
so the microlocalisation functor $M \mapsto M\otimes_RQ_T(R)$ preserves
pullbacks, and in particular, intersections. Part (e) follows.
\end{proof}

\subsection{Constructing normal elements}\label{ConstrNormal}
Let $I$ be a right ideal of $R$. Using Lemma \ref{MicMod}(c),
we can and will identify $Q_T(I)$ with a right ideal of $Q_T(R)$. By
Lemma \ref{MicMod}(d) this identification respects filtrations, and by
Lemma \ref{MicMod}(a) the associated graded ideal $\gr Q_T(I)$ is just
the localised right ideal $(\gr I)_T$ of $\gr Q_T(R) = (\gr R)_T$.

\begin{prop} Let $I$ be a two-sided of $R$ and suppose that there
exists a central regular homogeneous element $X \in \gr R$ such
that the localised ideal $(\gr I)_T$ of $(\gr R)_T$ is generated by $X$:
\[(\gr I)_T = X \cdot (\gr R)_T.\]
Then there exists a \emph{normal} element $w \in Q_T(R)$ such
that $Q_T(I) = w\cdot Q_T(R)$.
\end{prop}
\begin{proof} Choose any $w \in Q_T(I)$ such that $\gr w = X$.
Then the right ideal $w\cdot Q_T(R)$ is contained in $Q_T(I)$
and their graded ideals are equal by assumption. Because the
filtration on $Q_T(R)$ is complete, it follows that $Q_T(I) = w\cdot Q_T(R)$.

The Ore localisation $I_S$ is a \emph{two-sided} ideal of $R_S$
because $R$ is noetherian \cite[Proposition 2.1.16]{MR}. By
Lemma \ref{MicMod}(b), $Q_T(I)$ is the closure of $I_S$ inside
$Q_T(R)$ and is hence a two-sided ideal of $Q_T(R)$.

Since $X = \gr w$ is central and regular in $(\gr R)_T$, and
the filtration on $Q_T(R)$ is complete, the fact that $w$ is
a normal element in $R$ will follow from the following rather
general lemma.
\end{proof}

\begin{lem} Let $R$ be a complete filtered ring and $w\in R$.
Suppose that $wR$ is a two-sided ideal of $R$ and that $\gr w$ is
a central regular element of $\gr R$. Then $w$ is a regular
normal element in $R$.
\end{lem}
\begin{proof} Because $\gr w$ is a regular element of $\gr R$, $w$
must be a regular element of $R$. Since $Rw \subseteq wR$, for
every $r\in R$ there exists $\sigma(r)\in R$ such that $rw = w\sigma(r)$.
Since $w$ is regular, $r \mapsto \sigma(r)$ is an injective ring
endomorphism of $R$. We will show that $\sigma$ is surjective,
which will complete the proof.

Let $r \in R$ be nonzero, so that $\sigma(r)$ is nonzero.
Since $\gr w$ is central and regular,
\[\gr r \gr w = \gr(rw) = \gr(w\sigma(r)) = \gr w \gr \sigma(r) = \gr \sigma(r) \gr w\]
and therefore $\gr \sigma(r) = \gr r$ for any nonzero $r\in R$.
Now let $s \in R$ be a nonzero element of degree $n$. Set $r_n := s$, so that
\[s \equiv \sigma(r_n) \mod F_{n-1}R.\]
Set $r_{n-1} = s - \sigma(r_n) \in F_{n-1}R$, so that
\[s - \sigma(r_n) \equiv \sigma(r_{n-1}) \mod F_{n-2}R.\]
Continuing this process, we can construct a sequence of elements
$r_n, r_{n-1}, r_{n-2},\ldots $ of $R$ such
that $r_i \in F_iR$ for all $i\leq n$. Because $R$ is complete,
the infinite sum $\sum_{k=0}^\infty r_{n-k}$ converges to an
element $r$ of $R$ and $\sigma(r) = s$ by construction.
The result follows.
\end{proof}

\section{A control theorem for reflexive ideals}
In this section we state and prove our main result.
\subsection{Microlocalisation of Frobenius pairs}
\label{MicFrob}
Let $(A,A_1)$ be a Frobenius pair. Because $B = \gr A$ is noetherian
and the filtration on $A$ is complete, the remarks made in
$\S\ref{MicLocBasic}$ show that we may apply the theory
developed in $\S\ref{MicroLoc}$.

If $Z$ is a nonzero homogeneous element of $B$, then $T :=
\{1,Z,Z^2,\ldots\}$ is an Ore set in $B$ consisting of regular
homogeneous elements, since $B$ is a commutative domain by
assumption. By abuse of notation, we will denote the corresponding
microlocalisation $Q_T(A)$ by $Q_Z(A)$.

It turns out that Frobenius pairs are stable under microlocalisation.

\begin{prop}
Let $(A,A_1)$ be a Frobenius pair and $Z \in B$ be a nonzero
homogeneous element. 
\begin{enumerate}[{(}a{)}]
\item 
Then $(Q_Z(A), Q_{Z^p}(A_1))$ is also a Frobenius pair.
\item 
If $\mathbf{a}$ is a source of derivations for $(A,A_1)$, then it is
also a source of derivations for $(Q_Z(A), Q_{Z^p}(A_1))$.
\item 
Suppose $B$ is a UFD. If $(A,A_1)$ satisfies the derivation hypothesis,
then so does $(Q_Z(A), Q_{Z^p}(A_1))$.
\end{enumerate}
\end{prop}
\begin{proof} (a) By Proposition \ref{MicRings}(c), we can identify
$A$ with its image in $Q_Z(A)$. We will also identify $\gr A_1$ with
its image $B_1$ in $B$. By Definition \ref{FrobPair}(iii), $Z^p$ lies
in $B_1$ so the microlocalisation $Q_{Z^p}(A_1)$ makes sense.

Let $T$ and $T_1$ denote the multiplicatively closed sets in $B$ and
$B_1$ generated by $Z$ and $Z^p$ and let $S$ and $S_1$ be the
corresponding right Ore sets in $A$ and $A_1$. Clearly $S_1 \subseteq S$,
so the Ore localisation $(A_1)_{S_1}$ naturally embeds into $A_S$.
Moreover, using Lemma \ref{LiftOre} we see that
\[ (F_nA_S) \cap (A_1)_{S_1} = F_n(A_1)_{S_1}\]
for all $n\in\mathbb{Z}$, which means that the filtration on $(A_1)_{S_1}$
induced from $A_1$ coincides with the subspace filtration induced from $A_S$.
Passing to completions we see that $Q_{Z^p}(A_1)$ can be identified
with a closed subalgebra of $Q_Z(A)$. Moreover, one can easily check that
\[ (F_nQ_ZA) \cap Q_{Z^p}(A_1) = F_nQ_{Z^p}(A_1)\]
for all $n$. Hence Definition \ref{FrobPair}(i) is satisfied
for the new pair $(Q_Z(A), Q_{Z^p}(A_1))$.

Next, $\gr Q_Z(A) \cong B_Z$ and $\gr Q_{Z^p}(A_1) \cong
(B_1)_{Z^p}$ by Proposition \ref{MicRings}(e). Since $B_Z$ is a
commutative noetherian domain and $(B_Z)^{[p]} = B^{[p]}_{Z^p}
\subseteq (B_1)_{Z^p}$, Definitions \ref{FrobPair}(ii, iii) are
satisfied.

Finally, $B_Z \cong B_{Z^p}$ because $Z$ is becomes a unit
when $Z^p$ gets inverted. Hence
\[B_Z = B_{Z^p} = \bigoplus_{\alpha\in\mathbb{N}_t^p} (B_1)_{Z^p}
\mathbf{y}^\alpha = \bigoplus_{\alpha\in\mathbb{N}_t^p} (B_Z)_1
\mathbf{y}^\alpha,\]
which shows that Definition \ref{FrobPair}(iv) is inherited by $B_Z$.

(b) Let $a\in A$ and let the integers $k,n$ be such that $[a,F_kA]
\subseteq F_{k-n}A$. For any $y \in A$ and $s \in S$ we have
\[[a,ys^{-1}] = [a,y]s^{-1} - ys^{-1}[a,s]s^{-1},\]
which together with Lemma \ref{LiftOre} implies that
\[[a,F_kA_S] \subseteq F_{k-n}A_S.\]
Now $F_kQ_Z(A)$ is the closure of $F_kA_S$ in $Q_Z(A)$ by
Proposition \ref{MicRings}(b) and the bracket operation $[a,-]$ is continuous, so
\[[a, F_kQ_Z(A)] \subseteq F_{k-n}Q_Z(A).\]
A similar argument shows that if $[a,F_kA_1] \subseteq F_{k-n}A$,
then
\[[a, F_kQ_{Z^p}(A_1)] \subseteq F_{k-n}Q_Z(A).\]
Part (b) follows.

(c) Let $X, Y$ be homogeneous elements of $B_Z$ and suppose
that $Y$ lies in the $\mathbf{a}$-closure of $XB_Z$ for all
$\mathbf{a} \in \mathcal{S}(Q_Z(A),Q_{Z^p}(A_1))$.

Let $\mathbf{a}$ be a source of derivations for $(A,A_1)$. Note that
the derivation $D_r$ of $\gr Q_Z(A) \cong B_Z$ induced by the
element $a_r \in Q_Z(A)$ coincides with the extension to $B_Z$ of
the derivation $\{a_r,- \}_{\theta(a_r)}$ of $B$ induced by $a_r \in
A$.

Because $Y$ lies in the $\mathbf{a}$-closure of $XB_Z$,
$D_r(Y) \in XB_Z$ for all $r \gg 0$. We can find an integer
$n$ such that $Z^{p^n}Y \in B$. Then
\[D_r(Z^{p^n}Y) \in XB_Z \cap B\]
for all $r \gg 0$. Since $B_Z$ is a flat $B$-module, $XB_Z \cap B$
is a reflexive ideal of $B$ by Proposition \ref{ReflIdeals}(b).
Since $B$ is a UFD, Lemma \ref{UFD} implies that $XB_Z \cap B = X'B$
for some homogeneous element $X' \in B$. Hence
\[D_r(Z^{p^n}Y) \in X'B\] for all $r \gg 0$ and therefore
$Z^{p^n}Y$ lies in the $\mathbf{a}$-closure of $X'B$ for
any source of derivations $\mathbf{a}$ for $(A,A_1)$.
Because $(A,A_1)$ satisfies the derivation hypothesis, it follows that
$\mathcal{D}(Z^{p^n}Y) \subseteq X'B$. By Proposition
\ref{Derivations}(b), the localised $B_Z$-module $\mathcal{D}_Z$
can be identified with the set of all $(B_1)_{Z^p}$-linear
derivations of $B_Z$. But $\mathcal{D}_Z(Y) \subseteq XB_Z$
and part (c) follows.
\end{proof}

\subsection{Applying Theorem \ref{ContNormal}}\label{IoverJ}
We can now use the Control Theorem for normal elements to deduce
some information about arbitrary two-sided ideals. Recall the
definition of pseudo-null modules from $\S\ref{PseudoNull}$.

\begin{thm}
Let $(A,A_1)$ be a Frobenius pair satisfying the 
derivation hypothesis, such that $B$
and $B_1$ are UFDs. Let $I$ be a two-sided ideal of $A$ and $J =
(I\cap A)\cdot A_1$. Then $\gr I/ \gr J$ is pseudo-null.
\end{thm}
\begin{proof}
The right ideal $J$ is clearly contained in $I$, and we have
the following chain of inclusions of graded ideals in $B$:
\[\gr J \subseteq \gr I \subseteq \overline{\gr I} \subseteq \gr R,\]
where $\overline{\gr I}$ denotes the reflexive closure of $\gr I$
in $B$ defined in $\S \ref{ReflIdeals}$. Since $B$ is a UFD,
$\overline{\gr I} = XB$ for some homogeneous element
$X \in B$ by Lemma \ref{UFD}.

Let $Z$ be a nonzero homogeneous element of $B$ such that
$ZX \in \gr I$, and consider the microlocalisations
$A' := Q_Z(A)$ and $A_1' := Q_{Z^p}(A_1)$. By construction,
$(\gr I)_Z = X\cdot B_Z$, so the two-sided ideal $I' := Q_Z(I)$
of $A'$ is generated by a normal element $w$ of $A'$ by
Proposition \ref{ConstrNormal}. Because the Frobenius pair
$(A',A_1')$ satisfies the derivation hypothesis 
by Proposition \ref{MicFrob}(c),
the ideal $I' = wA'$ is controlled by $A_1'$ by Theorem \ref{ContNormal}:
\[I' = (I' \cap A_1')\cdot A'.\]
By Lemma \ref{FreeMod}, $A = \bigoplus_{\alpha\in [p-1]^t} \mathbf{u}^\alpha A_1$
and $A' = \bigoplus_{\alpha\in [p-1]^t} \mathbf{u}^\alpha A_1'$;
note that the \emph{same} generators occur in both expressions. Hence
\[A' = A\cdot A_1'\quad\mbox{and}\quad I' = I\cdot A' = I\cdot A\cdot A_1' = I\cdot A_1'.\]
Because $A$ is a finitely generated $A_1$-module,
Lemma \ref{MicMod}(e) implies that
\[Q_{Z^p}(I) \cap Q_{Z^p}(A_1) = Q_{Z^p}(I\cap A_1)\]
or equivalently, $(I \cdot A_1') \cap A_1' = (I\cap A_1)\cdot A_1'$. Hence
\[I' = (I' \cap A_1')\cdot A' = (I\cap A_1)\cdot A_1' \cdot A' = (I\cap A_1)A\cdot A'\]
and hence $I\cdot A' = J\cdot A'$. Passing to the graded ideals
and using Lemma \ref{MicMod}(a), we obtain $(\gr I)_Z = (\gr J)_Z$,
which means that $Z^n \gr I \subseteq \gr J$ for some integer $n$.

This holds for any $Z\in X^{-1}\gr I$, a finitely generated ideal in $B$. Hence
\[(X^{-1}\gr I)^m \subseteq \Ann_B(\gr I / \gr J)\]
for some integer $m$. But $B / X^{-1}\gr I \cong XB/I$ is
pseudo-null by Lemma \ref{UFD}, so $C:=B / (X^{-1}\gr I)^m$
is also pseudo-null by Lemma \ref{PseudoNull}. Since $\gr I / \gr J$
is a finitely generated $B$-module, it must be a quotient of a
direct sum of finitely many copies of $C$ and is therefore
pseudo-null, again by Lemma \ref{PseudoNull}.
\end{proof}

\subsection{A control theorem for reflexive ideals}
\label{Main}
We can now prove our main result. Recall from $\S\ref{IntroRefl}$
that a reflexive two-sided ideal is a two-sided ideal
which is reflexive as a right and left ideal.
See also the remark below.

\begin{thm} Let $(A,A_1)$ be a Frobenius pair satisfying 
the derivation hypothesis,
such that $B$ and $B_1$ are UFDs. Let $I$ be a reflexive two-sided
ideal of $A$. Then $I\cap A_1$ is a reflexive two-sided
ideal of $A_1$ and $I$ is controlled by $A_1$:
\[I = (I\cap A_1)\cdot A.\]
\end{thm}
\begin{proof} Retain the notation of
$\S\ref{IoverJ}$. Note that $A$ is a free right and left
$A_1$-module by Lemma \ref{FreeMod}. It follows from
Proposition \ref{ReflIdeals} that $I\cap A_1$ is a reflexive
ideal of $A_1$ and $J = (I\cap A_1)A$ is a reflexive right
ideal of $A$. It will clearly be enough to show that $I\subseteq J$.

Let $N$ be a right submodule of $I/J$. If we equip $N$ with
the subquotient filtration, then $\gr N$ is a submodule of
$\gr I/ \gr J$ and is hence pseudo-null by Theorem \ref{IoverJ}
and Lemma \ref{PseudoNull}. In particular, $E^0(\gr N) = E^1(\gr N) = 0$.

Since the filtration on $A$ is zariskian by the remarks made in
$\S\ref{MicLocBasic}$, there is a good filtration on $E^1(N)$ such
that $\gr E^1(N)$ is a subquotient of $E^1(\gr N)$ by
\cite[Proposition 3.1]{Bj}. Hence $\gr E^1(N) = 0$. It now follows
from \cite[Chapter II, \S 1.2, Lemma 9]{LV} that $E^1(N) = 0$.
Similarly $E^0(N) = 0$, and so $I/J$ is a pseudo-null right
$A$-module.

Let $x\in I$ and $(J : x) := \{a \in A : xa \in J\}$ be the
annihilator of the image of $x$ in $I/J$. By Proposition
\ref{PseudoNull}(a) we know that $(J : x)^{-1} = A$. Now
$J^{-1}x(J:x) \subseteq J^{-1}J \subseteq A$ so $J^{-1}x \subseteq
(J:x)^{-1} = A$. Hence $x \in \overline{J} = J$, as required.
\end{proof}

\begin{rem} If $B$ is a UFD, then $B$ is completely integrally
closed. In other words, $B$ is a maximal order  \cite[Proposition 5.1.3]{MR}.
It follows from \cite[X.2.1]{MaR} that $A$ is also a maximal
order. Now it is well known that a two-sided ideal $I$ of a
maximal order $A$ is reflexive as a right ideal if and only
if it is reflexive as a left ideal \cite[Proposition 5.1.8]{MR}.
\end{rem}

\section{Iwasawa algebras}
\label{IwaAlgs}

\subsection{The Campbell-Hausdorff series}
\label{CampHaus}
Following \cite{DDMS}, we define
\[\epsilon := \left\{ \begin{array}{l}2 \quad \mbox{ if }
\quad p=2 \\ 1 \quad \mbox{ otherwise.} \end{array} \right. \]
Recall \cite[\S 9.4]{DDMS} that a $\Zp$-Lie algebra $L$ is said
to be \emph{powerful} if $L$ is free of finite rank as a
module over $\Zp$ and $[L,L] \subseteq p^\epsilon L$.

Let $\Phi(X,Y)$ be the \emph{Campbell-Hausdorff series}
\cite[Definition 6.26]{DDMS}.

\begin{lem} Let $L$ be a powerful $\Zp$-Lie algebra,
$v,w\in L$ and $k \geq 0$. Then
\[\Phi(-v + p^kw,v) \equiv p^kw \mod p^{k+1}L.\]
\end{lem}
\begin{proof} By the definition of the Campbell-Hausdorff series,
\[\Phi(X,Y) = X + Y + \frac{1}{2}[X,Y] +
\sum_{n\geq 3}\sum_{\langle\mathbf{e}\rangle=n-1} q_\mathbf{e}(X,Y)_{\mathbf{e}}\]
where $\mathbf{e} = (e_1,\ldots,e_s)$ ranges over all possible
sequences of positive integers such that
$\langle\mathbf{e}\rangle := e_1 + \ldots + e_s = n-1$, $q_\mathbf{e}$
is a certain rational number and
\[(X,Y)_{\mathbf{e}} = [\cdots[\cdots[\cdots[[X,Y],\cdots,Y],X],\cdots,X],\cdots]\]
is a repeated Lie commutator depending on $\mathbf{e}$ of length $n$.
Fix the integer $n \geq 3$ and the sequence $\mathbf{e}$ for the time being.

Substitute $X = -v + p^kw$ and $Y = v$ into this repeated commutator
and expand: this gives a $\Zp$-linear combination of repeated
commutators of $v$ and $p^kw$ of length $n$. With the exception of
$[v,v,\ldots,v] = 0$, each one of these involves at least one $p^kw$
and hence is contained in $p^kL^n$, where $L^1 = L, L^2= [L,L], L^3
= [[L,L],L], \ldots$ is the lower central series of $L$.

Using the fact that $L$ is powerful, we deduce that
\[(-v+p^kw,v)_{\mathbf{e}} \in p^kL^n \subseteq p^{k + \epsilon(n - 1)}L.\]
Now as $n \geq 3$, $p^{\epsilon(n-1)}q_{\mathbf{e}} \in p^\epsilon \Zp$
by \cite[Theorem 6.28]{DDMS}, so
\[q_{\mathbf{e}}(-v+p^kw,v)_{\mathbf{e}} \in p^k q_{\mathbf{e}}
p^{\epsilon(n-1)}L \subseteq p^{k + \epsilon}L\]
for all $n \geq 3$ and all $\mathbf{e}$ such that $\langle \mathbf{e}\rangle = n - 1$. Hence
\[ \Phi(-v+p^kw,v) \equiv p^kw + \frac{p^k}{2}[w,v] \mod p^{k+\epsilon}L.\]
Now $\frac{p^k}{2}[w,v] \in p^{k+1}L$ since $[w,v] \in p^{\epsilon}L$
and the result follows.
\end{proof}

\subsection{The exponential map}\label{ExpMap}
Recall that there is an isomorphism
between the category of uniform pro-$p$ groups and group
homomorphisms and the category of powerful $\Zp$-Lie
algebras and Lie homomorphisms \cite[Theorem 9.10]{DDMS}.

If $L$ is a powerful $\Zp$-Lie algebra and $G$ is the corresponding
uniform pro-$p$ group, then there is a bijection
\[\exp : L \to G\]
which allows us to write every element of $G$ in the form $\exp(u)$
for some $u \in L$. The Campbell-Hausdorff series allows us to
recover the group multiplication in $G$ from the Lie structure
on $L$ \cite[Proposition 6.27]{DDMS}:
\[\exp(u) \cdot \exp(v) = \exp\Phi(u,v)\]
for all $u,v \in L$. We collect together some useful
properties of $\exp$ in the following

\begin{lem}
Let $L$ be a powerful $\Zp$-Lie algebra, $u,v \in L$ and $G$ be the
corresponding uniform pro-$p$ group. Then
\begin{enumerate}[{(}a{)}]
\item $\exp(mu) = \exp(u)^m$ for all $m\in\mathbb{Z}$,
\item $\exp(p^kL) = G^{p^k}$ for all $k \geq 0$,
\item if $u \equiv v \mod p^kL$ for some $k\geq 0$, then $\exp(u) \equiv \exp(v) \mod G^{p^k}$,
\item $\exp$ induces an $\Fp$-linear isomorphism $L/pL \to G/G^p$:
\[\exp(u + v) \equiv \exp(u)\exp(v) \mod G^p.\]
\end{enumerate}
\end{lem}
\begin{proof} For parts (a), (b) and (d) see the proof of
\cite[Theorem 9.8]{DDMS}. Now Lemma \ref{CampHaus} implies that $\Phi(-u,v) \in p^kL$, so
\[\exp(u)^{-1}\exp(v) = \exp(-u)\exp(v) = \exp\Phi(-u,v) \in \exp(p^kL) = G^{p^k}\]
and part (c) follows.
\end{proof}

Let $(g,h) = g^{-1}h^{-1}gh$ denote the group commutator of $g,h\in G$.

\begin{prop}
Let $u \in L$ be such that $[u,L] \subseteq p^kL$ for some $k\geq \epsilon$. Then
\[ (\exp(u), \exp(v)) \equiv \exp([u,v]) \mod G^{p^{k+1}}\]
for all $v\in L$. In particular, $(\exp(u), G) \subseteq G^{p^k}$.
\end{prop}
\begin{proof} We can compute the conjugate $\exp(-u)\exp(-v)\exp(u)$
in $G$ using \cite[Exercise 6.12]{DDMS}: $\exp(-u)\exp(-v)\exp(u) = \exp(-z)$, where
\[z := v.\exp(\ad(u)) = v + [v,u] + \frac{1}{2}[[v,u],u] +
\frac{1}{6}[[[v,u],u],u] + \cdots \in L.\] Now $\exp(p^k) = 1 + p^k
+ \frac{1}{2}p^{2k} + \ldots \equiv 1 + p^k \mod p^{k+1}\Zp$ and $L
\cdot \ad(u)\subseteq p^kL$, so
\[\frac{v \cdot \ad(u)^n}{n!} \in \frac{p^{kn}}{n!}L \subseteq p^{k+1}L\]
for all $n \geq 2$. Hence $z = v + p^kw$ for some $w \in L$ such that
$p^kw \equiv [v,u] \mod p^{k+1}L$.

Applying Lemma \ref{CampHaus}, we deduce that
\[ \Phi(-z,v) = \Phi(-v - p^kw,v) \equiv -p^kw \equiv [u,v] \mod p^{k+1}L.\]
Using Lemma \ref{ExpMap}(c), we finally obtain
\[ (\exp(u),\exp(v)) = \exp(-z)\cdot\exp(v) =
\exp \Phi(-z,v) \equiv \exp([u,v]) \mod G^{p^{k+1}},\]
as required.
\end{proof}

\subsection{Subalgebras and subgroups}
\label{SubAlgGrp}
From now on, we will assume that $L$ is a powerful $\Zp$-Lie
algebra of rank $d$ and we will fix a subalgebra $L_1$ of $L$
which contains $pL$. We can find a subset $\{v_1,\ldots,v_d\}$
of $L$ such that
\begin{itemize}
\item $\{v_i + pL : 1 \leq i \leq d\}$ is an $\Fp$-basis for $L/pL$, and
\item $\{v_i + L_1 : 1 \leq i \leq t\}$ is an $\Fp$-basis for $L/L_1$ for some $t$.
\end{itemize}
In many interesting cases, $L_1$ will in fact be equal to $pL$.

\begin{lem} Let $G = \exp(L)$ be the uniform pro-$p$ group corresponding
to $L$, let $G_1 = \exp(L_1)$ and let $g_i = \exp(v_i)$ for all $i$. Then
\begin{enumerate}[{(}a{)}]
\item $G_1$ is a \emph{subgroup} of $G$,
\item $\{g_1,\ldots,g_d\}$ is a topological generating set for $G$, and
\item $\{g_1^p,\ldots,g_t^p,g_{t+1},\ldots,g_d\}$ is a topological generating set for $G_1$.
\end{enumerate}
\end{lem}
\begin{proof}
(a) This is not entirely trivial, since $\exp(M)$ doesn't have to be
a subgroup of $G$ for arbitrary subalgebras $M$ of $L$ --- see
\cite{Ilani}. However, $\exp(u)\exp(v) \equiv\exp(u+v) \mod G^p$ for
all $u,v\in L$ by Lemma \ref{ExpMap} and $G^p = \exp(pL) \subseteq
\exp(L_1) = G_1$, so $xy \in G_1$ for all $x,y \in G_1$ and $G_1$
\emph{is} a subgroup.

(b) Let $M$ be the $\Zp$-submodule of $L$ generated by
$\{v_1,\ldots,v_d\}$. Because
\[M + pL = L\]
by assumption, $M=L$ by Nakayama's Lemma and hence $\{v_1,\ldots,v_d\}$ is
a $\Zp$-basis for $L$ since $L$ has rank $d$. Part (b) now
follows from \cite[Theorem 9.8]{DDMS}.

(c) By \cite[Theorem 3.6(iii)]{DDMS} and part (b), $\{g_1^p,\ldots,g_d^p\}$
is a topological generating set for $G^p$. Since $\{g_{t+1}G^p,\ldots,g_dG^p\}$
is a basis for $G_1/G^p$ by Lemma \ref{ExpMap}(d), $\{g_1^p,\ldots,g_t^p,g_{t+1},\ldots,g_d\}$
must be a topological generating set for $G_1$, as required.
\end{proof}

\subsection{The group algebra of a uniform pro-$p$ group}
\label{GpAlg} Let $G$ is a uniform pro-$p$ group and let $K$ be a
field of characteristic $p$. Let $J$ be the augmentation ideal of
the group algebra $K[G]$ of $G$. If $\{g_1,\ldots,g_d\}$ is a
topological generating set for $G$ and set $b_i := g_i - 1$ for all
$i=1,\ldots,d$, then these elements all lie in $J$.

\begin{prop} The associated graded ring of $K[G]$ with respect
to the $J$-adic filtration is isomorphic to the polynomial
algebra $K[y_1,\ldots,y_d]$.
\end{prop}
\begin{proof} As in the proof of \cite[Theorem 7.22]{DDMS}, the
$b_i$'s commute modulo $J^3$. We can therefore define a $K$-algebra
homomorphism $\varphi : K[y_1,\ldots, y_d] \to \gr K[G]$ by
setting $\varphi(y_i) = b_i + J^2$. When the field $K$ is
$\Fp$, \cite[Theorem 7.24]{DDMS} implies that $\varphi$ is
an isomorphism. The general case now follows, using a
simple ``extension of scalars" argument.
\end{proof}

From now on we will identify $K[y_1,\ldots,y_d]$ with $\gr K[G]$
via the map $\varphi$. For each $\alpha\in\mathbb{N}^t$, let
$\mathbf{b}^\alpha := b_1^{\alpha_1}\cdots b_d^{\alpha_d} \in K[G]$ and define
\[\mathcal{M} := \{\mathbf{b}^\alpha : \alpha \in \mathbb{N}^d\}.\]
Writing $|\alpha| := \alpha_1 + \ldots + \alpha_d$, we can define
\[\mathcal{M}_{<n} := \{\mathbf{b}^\alpha \in \mathcal{M} : |\alpha| < n\}\]
for each $n \geq 0$, the subsets $\mathcal{M}_{=n}$ and
$\mathcal{M}_{\geq n}$ being defined similarly.

\begin{cor} $K[G] = J^n \oplus K[\mathcal{M}_{<n}]$ for all $n\geq 0$.
\end{cor}
\begin{proof} The above proposition implies that $J^{n+1} = J^n \oplus K[\mathcal{M}_{=n}]$
for all $n\geq 0$. The corollary follows from this by an easy induction.
\end{proof}

\subsection{Subgroups}
\label{GrKG1}
Recall the notation of \S\ref{SubAlgGrp}, so that $\{g_1^p,\ldots,g_t^p,g_{t+1},\ldots,g_d\}$
is a topological generating set for $G_1$. Now define
\[\mathcal{N} := \{\mathbf{b}^\alpha\in\mathcal{M} : p|\alpha_i \quad\mbox{for all}\quad i\leq t\}\]
and note that the $K$-linear span $K[\mathcal{N}]$ of
$\mathcal{N}$ is contained in $K[G_1]$.

\begin{lem}\begin{enumerate}[{(}a{)}]
\item $K[\mathcal{N}]$ is dense in $K[G_1]$ with respect to the $J$-adic topology.
\item $K[\mathcal{N}] \cap J^n = K[\mathcal{N} \cap \mathcal{M}_{\geq n}]$, for all $n \geq 0$.
\item The image of $\gr K[\mathcal{N}]$ inside $\gr K[G]$ is equal to $K[y_1^p,\ldots,y_t^p,y_{t+1},\ldots,y_d].$
\end{enumerate}
\end{lem}
\begin{proof} (a) Let $x \in G_1$ and $n\geq 0$. It will be enough
to show that $x \equiv y \mod J^n$ for some $y\in K[\mathcal{N}]$.
Since $G/G^{p^n}$ is a finite powerful $p$-group, by
\cite[Corollary 2.8]{DDMS} we can find non-negative integers
$\lambda_1,\ldots,\lambda_d$ such that
\[x = g_1^{\lambda_1}\cdots g_d^{\lambda_d} u\]
for some $u\in G^{p^n}$. Considering the image of $x$ in $G/G^p$ and
using the fact that $x \in G_1$, we see that $\lambda_i$ is divisible
by $p$ for all $i \leq t$. Write $\lambda_i = p\mu_i$ for
some $\mu_i \in \mathbb{N}$, for each $i\leq t$. Let
$y := g_1^{\lambda_1}\cdots g_d^{\lambda_d}$; then
\[y = (1 + b_1^p)^{\mu_1}\cdots (1+b_t^p)^{\mu_t}(1+b_{t+1})^{\lambda_{t+1}}\cdots (1+b_d)^{\lambda_d} \in K[\mathcal{N}].\]
Because $G^{p^n} - 1 \subseteq J^n$, the element $u$ is congruent
to $1$ modulo $J^n$, and hence
\[ x = yu \equiv y \mod J^n\]
as required.

(b) It will be enough to show that $K[\mathcal{N}] \cap J^n
\subseteq K[\mathcal{N} \cap \mathcal{M}_{\geq n}]$, so let $a \in
K[\mathcal{N}] \cap J^n$. We can decompose $a$ uniquely as $a = b +
c$, where $b \in K[\mathcal{N} \cap \mathcal{M}_{<n}]$ and $c \in
K[\mathcal{N} \cap \mathcal{M}_{\geq n}]$. Now $c \in
K[\mathcal{M}_{\geq n}] \subseteq J^n$ so $b = a - c \in J^n \cap
K[\mathcal{M}_{<n}] = 0$ by Corollary \ref{GpAlg}. Hence $a = c \in
K[\mathcal{N} \cap \mathcal{M}_{\geq n}]$, as required.

(c) This follows immediately from part (b).
\end{proof}

\subsection{Completed group algebras}
\label{IwaFrobPair}
Let $H$ be a compact $p$-adic analytic group.
The \emph{completed group algebra} $KH$ is by definition the inverse
limit
\[ KH  :=  \underleftarrow{\lim} \hspace{2pt} K[H/N], \]
as $N$ runs over all the open normal subgroups of $H$. When the
field $K$ is finite, this algebra is sometimes called
the \emph{Iwasawa algebra} of $H$.

\begin{prop}
Let $A := KG$ and $A_1 := KG_1$. Then $(A,A_1)$ is a
Frobenius pair.
\end{prop}
\begin{proof}
For each open normal subgroup $N$ of $G$, let $w_{N,G}$ be
the kernel of the natural map from $K[G]$ to $K[G/N]$. By
the proof of \cite[Lemma 7.1]{DDMS}, this family of ideals
of $K[G]$ is cofinal with the powers of the augmentation
ideal $J = w_{G,G}$. Therefore $A$ is isomorphic to the
completion of $K[G]$ with respect to the $J$-adic filtration
on $K[G]$. Let $(F_nA)$ be the associated filtration on $A$; explicitly,
\[F_nA := \left\{ \begin{array}{l}\overline{J^{-n}} \quad \mbox{ if }
\quad n \leq 0 \\ A \quad \mbox{ otherwise.} \end{array} \right. \]
In this way, $A$ becomes a complete filtered $K$-algebra, and
\[B:=\gr A \cong \gr K[G] \cong K[y_1,\ldots,y_d]\]
is a commutative noetherian domain, by Proposition \ref{GpAlg}.

Now if $N$ is an open normal subgroup of $G$, then $N\cap G_1$ is an
open normal subgroup of $G_1$ and
\[w_{N,G} \cap K[G_1] = w_{N\cap G_1,G_1}.\]
Conversely, if $N_1$ is an open normal subgroup of $G_1$, then we
can find an open normal subgroup $N$ of $G$ such that $N \cap G_1
\subseteq N_1$, so that
\[w_{N_1,G_1} \supseteq w_{N \cap G_1,G_1} = w_{N,G} \cap K[G_1].\]
Hence the subspace topology on $K[G_1]$ induced from the
$J$-adic topology on $K[G]$ coincides with the natural
topology on $K[G_1]$ used in the definition of $A_1$.
We may therefore identify $A_1$ with the closure of $K[G_1]$
inside $A$. In this way $A_1$ becomes a closed subalgebra of $A$.

Finally, Lemma \ref{GrKG1} implies that the image of $\gr A_1 \cong
\gr K[G_1] \cong \gr K[\mathcal{N}]$ inside $\gr A$ can be
identified with the subalgebra $B_1 :=
K[y_1^p,\ldots,y_t^p,y_{t+1},\ldots,y_d]$ of $B$. This clearly
contains $B^{[p]}$ and moreover
\[ B = \bigoplus_{\alpha\in [p-1]^t} B_1 \mathbf{y}^\alpha,\]
as required.
\end{proof}

\subsection{Sources of derivations for Iwasawa algebras}
\label{SourcesIwa}

\begin{prop}
Let $u \in L$ be such that $[u,L] \subseteq p^kL$ and
$[u,L_1] \subseteq p^{k+1}L$ for some $k \geq \epsilon$,
and let $a = \exp(u)$. Then
\begin{enumerate}[{(}a{)}]
\item $(a, G) \subseteq G^{p^k}$,
\item $(a, G_1) \subseteq G^{p^{k+1}}$,
\item $[a,F_nA]  \subseteq  F_{n - p^k + 1}A$ for all $n\in\mathbb{Z}$, and
\item $[a,F_nA_1]  \subseteq  F_{n - p^{k+1} + p}A$ for all $n\in\mathbb{Z}$.
\end{enumerate}
\end{prop}
\begin{proof} Parts (a) and (b) follow from Proposition \ref{ExpMap}:
\[\begin{array}{clclcll}
(a,G) &=& (\exp(u),\exp(L)) &\subseteq& \exp([u,L])G^{p^{k+1}} &
\subseteq& G^{p^k} \quad\mbox{and} \\

(a,G_1) &=& (\exp(u),\exp(L_1)) & \subseteq &
\exp([u,L_1])G^{p^{k+1}} &=& G^{p^{k+1}}.
\end{array}\]
(c) It is sufficient to prove this for non-positive values of $n$, since then
\[[a,F_nA] = [a,F_0A] \subseteq F_{-p^k+1}A\subseteq F_{n - p^k+1}A\]
for all $n \geq 0$.  Let $h \in G$ and set $b:=h-1$. Then
\[[a, b] = [a,h] = ha((a,h) - 1) \in K[G](G^{p^k} - 1) \subseteq J^{p^k}\]
by (a), so by induction we have
\[[a,b^m] = b[a,b^{m - 1}] + [a,b]b^{m-1} \in J^{p^k + m - 1}\]
for all $m \geq 0$. Therefore
\[ [a, \mathbf{b}^\alpha] = [a,b_1^{\alpha_1}]b_2^{\alpha_2}\cdots
b_d^{\alpha_d} + \cdots + b_1^{\alpha_1} \cdots b_{d-1}^{\alpha_{d-1}}
[a,b_d^{\alpha_d}] \in J^{|\alpha| + p^k - 1}\]
for all $\mathbf{b}^\alpha\in\mathcal{M}$. Now
$K[\mathcal{M}_{\geq -n}]$ is dense in $F_nA$, so
\[ [a,F_nA] = \overline{[g,K[\mathcal{M}_{\geq -n}]]} \subseteq
\overline{J^{-n + p^k - 1}} = F_{n - p^k + 1}A,\]
as required.

(d) Again, we may assume that $n \leq 0$. Let $h \in G_1$ and set $b
= h-1$. Then
\[[a, b] = [a,h] = ha((a,h) - 1) \in K[G](G^{p^{k+1}} - 1) \subseteq J^{p^{k+1}}\]
by (b). Hence in the notation of $\S\ref{GrKG1}$, $[a,b_i^{pm}] \in
J^{pm + p^{k+1} - p}$ for all $i \leq t$ and $[a, b_i^m] \in
J^{m + p^{k+1} - 1} \subseteq J^{m + p^{k+1} - p}$ for all $i > t$,
whenever $m\geq 0$. We can now deduce as in part (a) that
\[[a,\mathbf{b}^\alpha] \in J^{|\alpha| + p^{k+1} - p}\]
for all $\mathbf{b}^\alpha \in \mathcal{N}$, or equivalently,
$[a, \mathcal{N} \cap \mathcal{M}_{\geq -n}] \subseteq J^{-n + p^{k+1}-p}$.
Part (d) now follows because $K[\mathcal{N} \cap \mathcal{M}_{\geq -n}]$
is dense in $F_nA_1$ by Lemma \ref{GrKG1}.
\end{proof}

\begin{cor} Let $u \in L$ be such that $[u,L] \subseteq p^kL$ and $[u,L_1]
\subseteq p^{k+1}L$ for some $k \geq \epsilon$, and let $a = \exp(u) \in G$.
Then $\mathbf{a} = \{a,a^p,a^{p^2}, \ldots\}$ is a source of derivations
for the Frobenius pair $(A,A_1)$.
\end{cor}
\begin{proof} For all $r \geq 0$, $[p^ru,L] \subseteq p^{r+k}L$ and
$[p^ru,L_1] \subseteq p^{r+k+1}L$. Now let $\theta(a^{p^r}) = p^{r+k} - 1$
and $\theta_1(a^{p^r}) = p\theta(a^{p^r})$ and apply the proposition.
\end{proof}

In particular, if $G$ is a uniform pro-$p$ group and $g \in G$, then
$g = \exp(u)$ for some $u\in L$. Since $L$ is powerful, $[u,L] \subseteq
p^\epsilon L$ and $[u,pL] \subseteq p^{\epsilon + 1}L$. Hence
$(g, g^p, g^{p^2}, \ldots)$ is always a source of derivations for $(KG, KG^p)$.

\subsection{Computing the corresponding derivations}
\label{IwaDers}
Let $u\in L$ be such that for some $k \geq
\epsilon$, we have
\begin{itemize}
\item $[u,L]\subseteq p^kL$
\item $[u,L]\nsubseteq p^{k+1}L$, and
\item $[u,L_1]\subseteq p^{k+1}L$.
\end{itemize}
Note that if such a $k$ exists, then it is uniquely determined by
$u$. Moreover, if $L_1 = pL$, then the third condition automatically
follows from the first, and in this case such an integer $k$ always
exists for any non-central element $u$ of $L$.

We can now define a well-defined non-zero $\Fp$-linear map
\[\begin{array}{lllc}
 \rho_u :& L/L_1 &\to& L/pL \\
  & v + L_1 &\mapsto& \frac{1}{p^k}[u,v] + pL.
\end{array}\]
Let $a = \exp(u)$. Since $[a, F_nA] \subseteq F_{n - p^k + 1}A$ for
all $n\in\mathbb{Z}$ by Proposition \ref{SourcesIwa}(c),
$u$ induces a derivation
\[D_u := \{a,-\}_{p^k - 1}\]
of $B = K[y_1,\ldots,y_d]$ as in \S\ref{Source}. It turns out that
there is a very close connection between $D_u$ and $\rho_u$. Recall
from \S\ref{SubAlgGrp} that $\{v_i + L_1 : 1 \leq i \leq t\}$ is an
$\Fp$-basis for $L/L_1$, and $\{v_i + pL : 1 \leq i \leq d\}$ is an
$\Fp$-basis for $L/pL$.

\begin{thm} Let $(c_{ij})$ be the matrix of $\rho_u$ with
respect to these bases. Then
\[D_u(y_j) = \sum_{i=1}^d c_{ij} y_i^{p^k}\]
for all $j=1,\ldots, t$.
\end{thm}
\begin{proof} Choose $\lambda_{ij} \in [p-1]$ such that $c_{ij}$
is the reduction of $\lambda_{ij}$ modulo $p$. By the
definition of $c_{ij}$,
\[\frac{1}{p^k}[u,v_j] \equiv \sum_{i=1}^d \lambda_{ij} v_i \mod pL\]
for all $j=1,\ldots, t$.  Recall from \S\ref{SubAlgGrp} that
$g_i = \exp(v_i)$ for all $i$. By Lemma \ref{ExpMap}(d),
\[\exp\left(\frac{1}{p^k}[u,v_j]\right) \equiv \prod_{i=1}^d g_i^{\lambda_{ij}} \mod G^p.\]
Now by \cite[Theorem 3.6(iv)]{DDMS}, $g \mapsto g^{p^k}$
induces an isomorphism between $G/G^p$ and $G^{p^k}/G^{p^{k+1}}$.
Using Lemma \ref{ExpMap}(a), we see that
\[\exp([u,v_j]) = \exp\left(\frac{1}{p^k}[u,v_j]\right)^{p^k}
\equiv \prod_{i=1}^d g_i^{p^k\lambda_{ij}} \mod G^{p^{k+1}}.\]
We can now apply Proposition \ref{ExpMap} and deduce that
\[(a,g_j) \equiv \exp([u,v_j]) \equiv \prod_{i=1}^d g_i^{p^k\lambda_{ij}} \mod G^{p^{k+1}}.\]
Next, recall that $b_j = g_j - 1$ and consider the commutator
$[a,b_j]$ inside $K[G]$:
\[ [a,b_j] = [a,g_j] = g_ja((a,g_j) - 1) = g_j a \left( h_j
\prod_{i=1}^d g_i^{p^k\lambda_{ij}} - 1 \right)\]
for some $h_j \in G^{p^{k+1}}$. Since we're interested in
$\{a,-\}_{p^k-1}$, we only need to compute $[a,b_j]$ modulo
$J^{p^k + 1}$. Now $h_j - 1 \subseteq G^{p^{k+1}} - 1 \subseteq
J^{p^{k+1}} \subseteq J^{p^k + 1}$, so
\[h_j \equiv 1 \mod J^{p^k + 1}.\]
Because $g_ja \equiv 1 \mod J$, we can deduce that
\[ [a,b_j] \equiv \prod_{i=1}^d (1 + b_i^{p^k})^{\lambda_{ij}} - 1
\equiv \sum_{i=1}^d c_{ij}b_i^{p^k} \mod J^{p^k + 1}\]
for all $j=1,\ldots, t$. The result follows.
\end{proof}

\subsection{Verifying the derivation hypothesis}
\label{VerStar}
In a forthcoming paper \cite{AWZ}, we will prove the
following result.

\begin{thm}\cite[Theorem A]{AWZ}
Let $\Phi(\Zp)$ be the Chevalley $\Zp$-Lie algebra associated 
to a root system $\Phi$. 
Let $L$ be the Lie algebra $p^t\Phi(\Zp)$ for some $t\geq 1$
and $G$ be the corresponding uniform pro-$p$ group $\exp(L)$.
Suppose that $p\geq 5$ and that $p \nmid n+1$ if $\Phi$ has an 
indecomposable component of type $A_n$. Then $(KG,KG^p)$ 
satisfies the derivation hypothesis.
\end{thm}

For the time being, we only verify that the derivation hypothesis 
holds in the special case when $G$ is a congruence subgroup of
$\SL_2(\Zp)$.

\subsection{Congruence subgroups of $\SL_2(\Zp)$, $p\geq 3$}
\label{CongSubSL2}
Fix an integer $l\geq 1$ and let $L$ be the powerful Lie algebra
$\mathfrak{sl}_2(p^l \Zp)$. Thus $L$ has a basis
\[\left\{e=\begin{pmatrix} 0& p^l\\0&0\end{pmatrix},
f=\begin{pmatrix} 0& 0\\p^l&0\end{pmatrix},
h=\begin{pmatrix} p^l& 0\\0&-p^l\end{pmatrix}\right\}\]
satisfying the following relations:
\begin{itemize}
\item $[h,e] = 2p^l e$,
\item $[h,f] = -2p^l f$,
\item $[e,f] = p^l h$.
\end{itemize}
Let $\Gamma_l(\SL_2(\Zp))$ denote the \emph{l-th congruence subgroup}
of $\SL_2(\Zp)$:
\[\Gamma_l(\SL_2(\Zp)) := \ker\left(\SL_2(\Zp) \to \SL_2(\mathbb{Z}/p^l\mathbb{Z})\right).\]
It is well known that $G:=\exp(L)$ is isomorphic to $\Gamma_l(\SL_2(\Zp))$.
We let $L_1 = pL$, so that the corresponding subgroup $G_1$ is
just $G^p \cong \Gamma_{l+1}(\SL_2(\Zp))$.

We use the same variables $\{e,f,h\}$ for the generators of the
associated graded ring $B = \gr \FFp G$ and
hope that this will cause no confusion. Thus
\[ B = K[e,f,h] \quad\mbox{and}\quad B_1 = K[e^p,f^p,h^p].\]
Let $\{\frac{\partial}{\partial e}, \frac{\partial}{\partial f},
\frac{\partial}{\partial h} \}$ be the corresponding derivations,
which were constructed in \S\ref{Derivations}. Using Theorem
\ref{IwaDers} and Proposition \ref{Derivations}(b), we can
write down the derivations
\[D_{p^ru} = \{\exp(p^ru),-\}_{p^{r+l} - 1} : B \to B\]
generated by $u$ explicitly, for each $u\in \{e,f,h\}$:
\[\begin{array}{lll}
D_{p^re} &=& h^{p^{l+r}} \frac{\partial}{\partial f}
- 2e^{p^{l+r}} \frac{\partial}{\partial h}, \\
D_{p^rf} &=& -h^{p^{l+r}} \frac{\partial}{\partial e}
+ 2f^{p^{l+r}} \frac{\partial}{\partial h}. \\
D_{p^rh} &=& 2e^{p^{l+r}} \frac{\partial}{\partial e}
- 2f^{p^{l+r}} \frac{\partial}{\partial f}. \\
\end{array}
\]

\begin{prop} Let $l \geq 1$ and let $G = \exp(\mathfrak{sl}_2(p^l\Zp))$
as above. Then the Frobenius pair $(KG, KG^p)$
satisfies the derivation hypothesis.
\end{prop}
\begin{proof} Let $X,Y$ be homogeneous elements of $B$ and suppose
that $Y$ lies in the $\mathbf{a}$-closure of $XB$ for all
$\mathbf{a}\in\mathcal{S}(KG,KG^p)$. By Corollary \ref{SourcesIwa},
$(g, g^p, g^{p^2}, \ldots) \in \mathcal{S}(KG,KG^p)$ for all
$g \in G$, so we can find an integer $s$ such that
\[D_{p^ru}(Y) \in XB\]
for all $u \in \{e,f,h\}$ and all $r \geq s$. Consider the
$D_{p^re}$--equations for $r = s$ and $r = s+1$. Eliminating
the terms involving $\partial Y/ \partial f$ yields that
\[2 e^{p^{s+l}} \left( h^{p^{s+l}(p-1)} - e^{p^{s+l}(p-1)} \right)
\frac{\partial Y}{\partial h} \in XB,\]
and using similar operations with the $D_{p^rf}$-equations we have
\[2 f^{p^{s+l}} \left( h^{p^{s+l}(p-1)} - f^{p^{s+l}(p-1)} \right)
\frac{\partial Y}{\partial h} \in XB.\]
The coefficients of $\partial Y/\partial h$ appearing in the above
two equations are coprime, which allows us to deduce
\[\frac{\partial Y}{\partial h} \in XB.\]
Similar manipulations with the other equations show that
$\partial Y / \partial e$ and $\partial Y /\partial f$
also lie in $XB$. Hence $\mathcal{D}(Y) \subseteq XB$,
by Proposition \ref{Derivations}(b).
\end{proof}

\section{Ideals in Iwasawa algebras}

\subsection{Canonical dimension function}
\label{CanDimFun}
Let $A$ be a noetherian ring. We say that $A$ is
\emph{Gorenstein} if it has finite injective dimension on both
sides. For any finitely generated left (or right) $A$-module $M$,
the {\it $j$-number} or \emph{grade} of $M$ is defined to be
$$j(M) := \inf\{n\;|\; \Ext^n_A(M,A)\neq 0\}.$$
The ring $A$ is called {\it Auslander-Gorenstein\/} if it is
Gorenstein and it satisfies the {\it Auslander condition}:
\begin{enumerate}
\item[]
For every finitely generated left (respectively, right)
$A$-module $M$ and every positive integer $q$, one has $j(N)
\geq q$ for every finitely generated right (respectively, left)
$A$-submodule $N\subseteq \Ext^q_A(M,A)$.
\end{enumerate}
An {\it Auslander-regular} ring is a noetherian, Auslander-Gorenstein
ring which has finite global dimension. See \cite{Bj} for
some details. Note that a noetherian commutative regular algebra
is always Auslander-regular.  For any Auslander-Gorenstein
ring $A$, there is a {\it canonical dimension function} defined by
$$\Cdim (M)=\injdim (A)-j(M)$$
for all finitely generated left (or right) $A$-modules $M$
\cite[\S 5.3]{AB1}. This is a dimension function in the
sense of \cite[\S 6.8.4]{MR}.
Recall that a finitely generated $A$-module is said to be
\emph{pure} if $\Cdim(N) = \Cdim(M)$ for all nonzero submodules
$N$ of $M$. We will use the following nice observation of
Venjakob \cite[Lemma 4.12]{CSS}:

\begin{lem} Let $A$ be an Auslander-regular domain and
$I$ be a proper nonzero right ideal of $A$. Then $I$ is reflexive
if and only if $A/I$ is pure of grade $1$.
\end{lem}

\subsection{Crossed products}
\label{CrossProd}
Let $R$ be an Auslander-Gorenstein ring, $G$ be a finite group
and $S = R\ast G$ be a crossed product. We know by
\cite[Lemma 5.4]{AB2} that the restriction $M_{|R}$ to $R$
of any finitely generated $S$-module $M$ satisfies
\[\Cdim_S(M) = \Cdim_R(M_{|R}).\]
Hence $S$ is also Auslander-Gorenstein.

\begin{prop} $R$ has an ideal $I$ with $\Cdim_R(R/I) = n$
if and only if $S$ has an ideal $J$ with $\Cdim_S(S/J) = n$.
\end{prop}

\begin{proof} ($\Rightarrow$) Choose a set of units
$\{\overline{g} : g\in G \}$ in $S$ such that
$R \overline{g} = \overline{g} R$ inside $S$ and
$S = \bigoplus_{g\in G} R\overline{g}$. Then
$\alpha_g : r \mapsto \overline{g}^{-1}r\overline{g}$
is an algebra automorphism of $R$. Hence
$\Cdim_R(R / \alpha_g(I)) = \Cdim_R(R/I)$
for all $g \in G$. We set $I_0 := \bigcap_{g\in G} \alpha_g(I)$,
which is a $G$-invariant ideal in $R$. It follows
from the fact $I_0 \subseteq I$
that $\Cdim_R(R/I_0) \geq \Cdim_R(R/I)$. Since
\[R/I_0 \hookrightarrow \bigoplus_{g\in G} R / \alpha_g(I),\]
we actually have equality. Let us set $J = I_0\cdot S$; where $I_0$ is
$G$-invariant, $J$ is a twosided ideal in $S$ and by construction
\[ \Cdim_S(S/J) = \Cdim_R( (S/J)_{|R} ) = \Cdim_R(R/I).\]

($\Leftarrow$) We set $I := J\cap R$, which is a $G$-invariant ideal
in $R$. Then $S/IS = \bigoplus_{g\in G}(R\overline{g}/I\overline{g})$
and hence $(S/IS)_{|R} \cong (R/I)^{|G|}$.
Since $S/IS \twoheadrightarrow S/J$, we have
\[\Cdim_R(R/I) = \Cdim_R((S/IS)_{|R}) \geq \Cdim_R((S/J)_{|R}).\]
On the other hand $R/I \hookrightarrow (S/J)_{|R}$, so
we have equality and the result follows.
\end{proof}

\subsection{Proof of Theorem A}
\label{ProofThmA} We present a slightly more general version
of Theorem A this section. Let $\mathcal{L}(G)$ 
denote the $\Qp$-Lie algebra of $G$.

\begin{thm}
Let $\FFp$ be a field of characteristic
$p$. Suppose $G$ is a compact $p$-adic analytic group of
dimension $d$ such that $\mathcal{L}(G)$ is
split semisimple over $\Qp$. Suppose that $p\geq 5$, and that $p \nmid n$ in the case when $\mathfrak{sl}_n(\Qp)$ occurs as a direct summand of $\mathcal{L}(G)$.
Then $\FFp G$ has no two-sided ideals $I$ such that
\[\Cdim_{\FFp G} (\FFp G / I) = d - 1.\]
\end{thm}

\begin{proof} Note that $\FFp G$ is a crossed product of the
Auslander-Gorenstein ring $\FFp N$ with the finite group $G/N$,
for any open normal uniform subgroup $N$ of $G$. By
Proposition \ref{CrossProd}, we may replace $G$ by
any uniform pro-$p$ group $N$ having the same
$\Qp$-Lie algebra without affecting the conclusion of the Theorem.

By considering a suitable Chevalley basis, we can find a
sub $\Zp$-Lie algebra $L\subset\mathcal{L}(G)$ such that 
$L\cong p^t \Phi(\Zp)$ where $\Phi$ is the root system 
associated to $\Qp\otimes \mathcal{L}(G)$.
Now take $N$ to be the corresponding uniform pro-$p$
group $\exp(L)$. The $\Zp$-Lie algebra of $N^{p^k}$
is $p^kL$, so the Frobenius pair $(\FFp N^{p^k},\FFp N^{p^{k+1}})$
satisfies the derivation hypothesis for all $k\geq 0$ by Theorem \ref{VerStar}.

Suppose for a contradiction that $I$ is a two-sided ideal of $\FFp N$ such that
\[\Cdim_{\FFp N} (\FFp N / I) = d - 1.\]
By replacing $I$ by the inverse image of the largest pseudo-null
submodule of $\FFp N / I$ in $\FFp N$ we may assume that $\FFp N / I$
is pure. Note that $I$ is proper and \emph{nonzero}, since otherwise
$d= \Cdim_{\FFp N}(\FFp N) = d - 1$. It follows from Lemma \ref{CanDimFun}
that $I$ is a reflexive ideal of $\FFp N$. Applying Theorem \ref{Main}
repeatedly, we see that $I$ is controlled by $\FFp N^{p^k}$ for each $k$:
\[I = (I \cap \FFp N^{p^k})\cdot \FFp N.\]
Since $I$ is a proper ideal of $\FFp N$, we see that $I\cap \FFp
N^{p^k}$ must be contained in the maximal ideal $(N^{p^k} - 1)\cdot
\FFp N^{p^k}$ of $\FFp N^{p^k}$ for all $k\geq 0$. Hence
\[I \subseteq \bigcap_{k=0}^\infty ((N^{p^k} - 1) \cdot \FFp N) = 0,\]
a contradiction.
\end{proof}

\section{The case when $p=2$}

\subsection{Congruence subgroups of $\SL_2(\Zp), p = 2$}\label{p=2}
The reader might have wondered why we didn't just assume that $L_1 =
pL$ from \S\ref{SubAlgGrp} onwards. The reason is that the extra
generality allows us to be more flexible when choosing the
particular open subgroup of $G$ that we should try to "descend"
towards. The case of open subgroups in $\SL_2(\Zp)$ when $p=2$
should illustrate this flexibility: if $G = \Gamma_l(\SL_2(\Zp))$
and $p=2$, then $(KG,KG^p)$ does not satisfy the 
derivation hypothesis, but we can
circumvent this problem by going down to $G^p$ from $G$ in two
steps.

So assume that $p=2$ and fix $l \geq 2$. We choose the same
basis $\{e,f,h\}$ for $L_0 := \mathfrak{sl}_2(p^l\Zp)$ as
in \S\ref{CongSubSL2}, so that the following relations are satisfied:
\begin{itemize}
\item $[h,e] = p^{l+1} e$,
\item $[h,f] = -p^{l+1} f$,
\item $[e,f] = p^l h$.
\end{itemize}
Let $L_1 = pe\Zp \oplus pf\Zp \oplus h\Zp$ and let $L_2 = pL$.
The relations in $L_1$
\begin{itemize}
\item $[h,pe] = p^{l+1} (pe)$,
\item $[h,pf] = -p^{l+1} (pf)$,
\item $[pe,pf] = p^{l+2} h$
\end{itemize}
show that $L_1$ is a powerful $\Zp$-subalgebra of $L_0$ which
contains $L_2$. Moreover, $pL_1 \subseteq L_2$, so the pairs
$(L_0,L_1)$ and $(L_1,L_2)$ both satisfy the assumptions
made in \S\ref{SubAlgGrp}, and hence $(KG_0,KG_1)$ and
$(KG_1,KG_2)$ are Frobenius pairs, by Proposition \ref{IwaFrobPair}.
However the parameter $t$ equals $2$ in the first
case and $1$ in the second case.

\begin{prop} Let $G_i = \exp(L_i)$ for each $i=0,1,2$.
Then the Frobenius pairs $(KG_0,KG_1)$ and $(KG_1,KG_2)$
both satisfy the derivation hypothesis.
\end{prop}
\begin{proof} We first deal with the case $(KG_0,KG_1)$;
in this case, $B = K[e,f,h]$ and $B_1 = K[e^p,f^p,h]$. We observe that
\[\begin{array}{lll}
[e,L_0] \subseteq p^lL_0, \quad &[e,L_0] \nsubseteq p^{l+1}L_0,
\quad &[e,L_1] \subseteq p^{l+1}L_0, \\

[f,L_0] \subseteq p^lL_0, \quad &[f,L_0] \nsubseteq p^{l+1}L_0,
\quad &[f,L_1] \subseteq p^{l+1}L_0, \\

[h,L_0] \subseteq p^{l+1}L_0, \quad &[h,L_0] \nsubseteq p^{l+2}L_0,
\quad &[h,L_1] \subseteq p^{l+2}L_0.
\end{array}\]
By Theorem \ref{IwaDers}, we obtain three sets of derivations
of $B = \gr KG_0$ arising from sources of derivations of $(KG_0, KG_1)$:
\[\begin{array}{lll}
D_{p^re} &=& h^{p^{l+r}} \frac{\partial}{\partial f}, \\
D_{p^rf} &=& h^{p^{l+r}} \frac{\partial}{\partial e}, \\
D_{p^rh} &=& e^{p^{r+l+1}} \frac{\partial}{\partial e} -
f^{p^{r+l+1}} \frac{\partial}{\partial f}.
\end{array}
\]
Let $X,Y$ be homogeneous elements of $B$ and suppose that
$Y$ lies in the $\mathbf{a}$-closure of $XB$ for all
$\mathbf{a}\in\mathcal{S}(KG_0,KG_1)$; we can thus find
an integer $s$ such that
\[D_{p^ru}(Y) \in XB\]
for all $u \in \{e,f,h\}$ and all $r \geq s$. Eliminating
the terms involving $\partial Y / \partial f$ from the $D_{p^rh}$
equations for $r=s$ and $r=s+1$ shows that
\[f^{p^{s+l+1}}\left(f^{p^{s+l+1}(p-1)} + e^{p^{s+l+1}(p-1)}\right)
\partial Y / \partial f \in XB\]
Since $D_{p^se}(Y) = h^{p^{s+l}} \partial Y / \partial f \in XB$
and the coefficients of $\partial Y/\partial f$ are coprime,
\[\partial Y / \partial f \in XB.\]
Similarly $\partial Y/\partial e\in XB$, so the derivation hypothesis holds by
Proposition \ref{Derivations}(b).

Now consider the case $(KG_1, KG_2)$. Recycling notation,
let $\{e,f,h\}$ be the basis for $L_1$ considered above, so
that $\{e,f,ph\}$ is a basis for $L_2$, and the relations
\begin{itemize}
\item $[h,e] = p^{l+1} e$,
\item $[h,f] = -p^{l+1} f$,
\item $[e,f] = p^{l+2} h$
\end{itemize}
hold in $L_1$. The corresponding graded rings are $B = K[e,f,h]$
and $B_1 = K[e,f,h^p]$. Since
\[\begin{array}{lll}
[e,L_1] \subseteq p^{l+1}L_1, \quad &[e,L_1] \nsubseteq p^{l+2}L_1,
\quad &[e,L_2] \subseteq p^{l+2}L_1, \\

[f,L_1] \subseteq p^{l+1}L_1, \quad &[f,L_1] \nsubseteq p^{l+2}L_1,
\quad &[f,L_2] \subseteq p^{l+2}L_1,
\end{array}\]
Theorem \ref{IwaDers} gives us two sets of derivations of $B$
arising from sources of derivations of $(KG_1,KG_2)$:
\[\begin{array}{lll}
D_{p^re} &=& e^{p^{l+r+1}} \frac{\partial}{\partial h}, \\
D_{p^rf} &=& f^{p^{l+r+1}} \frac{\partial}{\partial h}. \\
\end{array}
\]
Let $X,Y$ be homogeneous elements of $B$ and suppose that $Y$ lies
in the $\mathbf{a}$-closure of $XB$ for all $\mathbf{a}\in
\mathcal{S}(KG_1,KG_2)$; we can thus find an integer $s$ such that
\[D_{p^ru}(Y) \in XB\]
for all $u \in \{e,f\}$ and all $r \geq s$. In particular,
$D_{p^re}(Y) = e^{p^{r+l+1}}\partial Y/\partial h$ and
$D_{p^rf}(Y) = f^{p^{r+l+1}}\partial Y/\partial h$ both
lie in $XB$. Since the coefficients of $\partial Y/\partial h$
are coprime, $\partial Y/\partial h \in XB$, so the 
derivation hypothesis holds.
\end{proof}

\begin{cor}
Let $K$ be a field of characteristic $2$ and suppose that $G$ is an
open subgroup of $\SL_2(\mathbb{Z}_2)$. Then $KG$ has no two-sided
ideals $I$ such that
\[\Cdim_{KG} (KG / I) = 2.\]
\end{cor}
\begin{proof} Follow the proof of Theorem A.
\end{proof}

\subsection{Proof of Theorem C}
Let $I$ be a prime ideal of $KG$. The dimension of $G$
is three, so the possible values for $c = \Cdim_{KG}(KG/I)$
when $I$ is a two-sided ideal of $KG$ are $0,1,2$ or $3$.
By Theorem \ref{ProofThmA} and Corollary \ref{p=2}, $c$
cannot be equal to $2$ and by \cite[Theorem A]{A}, $c$ cannot
be equal to $1$. Hence $c=0$, in which case $I$ is the maximal
ideal of $KG$ since $KG$ is local, or $c=3$ in
which case $I = 0$. \qed

\section*{Acknowledgments}
The authors thank Quanshui Wu for useful conversations. The first
author would like to thank Simon Wadsley and Sukhendu Mehrotra for
helpful comments, and the Department of Pure Mathematics at the 
University of Sheffield for financial support. Wei is supported by a research fellowship from the
China Scholarship Council and by the Department of Mathematics at
the University of Washington. Zhang is supported by the US National
Science Foundation and the Royalty Research Fund of the University
of Washington.

\end{document}